\pdfoutput=1

\documentclass[11pt]{amsart}
\usepackage{amssymb}
\usepackage{cite}
\newtheorem{theorem}{Theorem}[section]
\newtheorem{lemma}[theorem]{Lemma}
\newtheorem{corollary}[theorem]{Corollary} 
\theoremstyle{definition}
\newtheorem{definition}[theorem]{Definition}
\newtheorem{example}[theorem]{Example}

\newtheorem{notation}[theorem]{Notation} 
\theoremstyle{remark}
\newtheorem{remark}[theorem]{Remark}

\numberwithin{equation}{section}

\def\v#1#2#3#4{\sum_{\substack{#1=#2\\#1\ne #3}}^#4}
\def\nr#1#2#3#4{\|#1^#2-#3^#4\|}

\begin{document}

\title[Generalization of the Apollonius theorem for simplices]{Generalization of the Apollonius theorem for simplices and related problems}
\author[M. N. Vrahatis]{Michael N. Vrahatis}
\address{\kern-0.08cm Department of Mathematics, University of Patras, GR-26110 Patras, Greece}
\email{vrahatis@math.upatras.gr\,\, \&\, vrahatis@upatras.gr}

\begin{abstract}
The Apollonius theorem gives the length of a median of a triangle in terms of the lengths of its sides.
The straightforward generalization of this theorem obtained for \emph{m}-simplices in the \emph{n}-dimensional Euclidean space
for \emph{n} greater than or equal to \emph{m} is given.
Based on this, generalizations of properties related to the medians of a triangle are presented.
In addition, applications of the generalized Apollonius's theorem and the related to the medians results, are given
for obtaining:
(a) the minimal spherical surface that encloses a given simplex or a given bounded set,
(b) the thickness of a simplex that it provides a measure for the quality or how well shaped a simplex is, and
(c) the convergence and error estimates of the root-finding bisection method applied on simplices.
\end{abstract}

\maketitle

\thispagestyle{empty}

\section{Introduction}
Apollonius\footnote{Apollonius of Perga (\emph{c}.\ 240 -- \emph{c}.\ 190 B.C.), Greek mathematician, known as the ``Great Geometer''.}
gave a theorem known as the \emph{Apollonius theorem}\/ that relates
the length of a median of a triangle in terms of the lengths of its sides.
Specifically, he gave the following well known and widely used
theorem: 
\begin{theorem}[Apollonius' theorem]\label{ApolloniusTh1}
In any triangle $\Delta A B C$ with vertices $A$, $B$, $C$ and lengths of its sides $\ell (\,\,\overline{\!\!AB})$, $\ell (\,\,\overline{\!\!AC})$ and
$\ell (\,\,\overline{\!\!BC})$,
if\,\,\, $\overline{\!\!AD}$  is a median with length $\ell (\,\,\overline{\!\!AD})$, it holds that:
\begin{equation}\label{eq:Apollonius}
2\bigl[\ell (\,\,\overline{\!\!AB})^{2} + \ell (\,\,\overline{\!\!AC})^{2}\bigr] = \ell (\,\,\overline{\!\!BC})^{2} + 4\,\ell (\,\,\overline{\!\!AD})^{2}.
\end{equation}
\end{theorem}
The straightforward generalization of the concept of a triangle to the $n$\!\! -\!\! dimensional Euclidean space $\mathbb{R}^n$ is the simplex,
that is the smallest possible $n$-dimensional polytope, \emph{i.e.} the convex hull of a finite number of points~\cite[p.8]{Guggenheimer1977}.
In general, in geometry, a simplex is considered to be a generalization of the notion of a triangle to arbitrary dimensions.

Based on this notion we give the generalization of Apollonius' theorem for $m$-simplices in $\mathbb{R}^n$ ($n \geqslant m$) (see \S\ref{subsec:GenApo}).
Furthermore, we give generalizations obtained for simplices of well known properties related to the medians of triangles (see \S\ref{subsec:MedProp}).
Also, applications of the generalized Apollonius' theorem and the related results are given
for obtaining:
(a) the minimal spherical surface that encloses a given simplex or a bounded subset of $\mathbb{R}^n$ (see \S\ref{subsec:SimCov}),
(b) the thickness of a simplex that provides a measure for the quality or how well shaped a simplex is (see \S\ref{subsec:SimThick}) and
(c) the convergence and error estimates of the root-finding bisection method applied on simplices (see \S\ref{subsec:SimBis}).
The provided proofs of the proposed results are simple.

\section{Generalization of the Apollonius theorem for simplices}\label{sec:GenApo}
\subsection{Notation, definitions, and preliminary concepts}\label{subsec:Prelim}

\begin{definition}
For any positive integer $n$, and for any set of points $\{ {\upsilon}^k\}_{k=0}^n$
in some linear space which are affinely independent (\emph{i.e.} the vectors
$\{ {\upsilon}^k - {\upsilon}^0\}_{k=1}^n$
are linearly independent)
the convex hull
${\rm co}\{{\upsilon}^0,$ $ {\upsilon}^1, \ldots,$ ${\upsilon}^n \} =
[ {\upsilon}^0,$ ${\upsilon}^1, \ldots,$ ${\upsilon}^n ]$
is called the {\em $n$-simplex with vertices}
${\upsilon}^0,$ ${\upsilon}^1, \ldots, {\upsilon}^n$.
The convex hull of any nonempty subset of the $n + 1$ points that define an $n$-simplex is called a \emph{face}\/ of the simplex.
Faces are simplices themselves and they are also called {\em facets}.
Specifically, for each subset of $m+1$ elements
$\{ {\omega }^0, {\omega }^1, \ldots ,{\omega }^m \} \subset
\{ {\upsilon}^0, {\upsilon}^1, \ldots ,{\upsilon}^n \}, $
the $m$-simplex
$[ {\omega }^0, {\omega }^1, \ldots ,{\omega }^m ]$
is called an {\em $m$-face} of
$[ {\upsilon}^0,$ $ {\upsilon}^1, \ldots,{\upsilon}^n ].$
In particular, 0-faces are  vertices and 1-faces are edges.
If all the edges have the same length, the simplex is called \emph{regular}\/ or \emph{equilateral}\/ (\emph{cf.}~\cite{BlumenthalW1941}).
\end{definition}

\begin{remark}
The order in which the points $\{{\upsilon}^k\}_{k=0}^n$ are written in the lists determines an orientation.
On the other hand, the convex hull is independent of this order and thus the vertices can be permuted in a list.
\end{remark}

\begin{notation}\label{Not-3}
An $m$-simplex $\sigma^m$ in $\mathbb{R}^n$ ($n \geqslant m$)  with set of vertices
$\{ {\upsilon}^k\}_{k=0}^m$
is denoted by
$\sigma^m = [{\upsilon}^0, {\upsilon}^1, \ldots,{\upsilon}^m ]$ and its boundary by ${\vartheta}\kern0.01cm$${\sigma}^m$.
Also, we denote the $(m-1)$-simplex that determines the $i$-th $(m-1)$-face opposite to vertex ${\upsilon}^i$ of ${\sigma}^m$
by
$\sigma^{m}_{\neg i} = [{\upsilon}^0,$ $
{\upsilon}^1, \ldots,$ ${\upsilon}^{i-1}, {\upsilon}^{i+1}, \ldots,
{\upsilon}^m]$.
Similarly, we denote the $(m-2)$-simplex that determines the $j$-th $(m-2)$-face opposite to vertex ${\upsilon}^j$ of $\sigma^{m}_{\neg i}$,
for $j \neq i$,
by
$\sigma^{m}_{\neg ij} = [{\upsilon}^0,$ $
{\upsilon}^1, \ldots,$ ${\upsilon}^{i-1}, {\upsilon}^{i+1}, \ldots,$
${\upsilon}^{j-1}, {\upsilon}^{j+1}, \ldots,
{\upsilon}^m]$.
\end{notation}

\begin{remark}
Obviously, it holds that ${\vartheta}\kern0.01cm$${\sigma}^m = \bigcup _{i=0}^{m} \sigma^{m}_{\neg i}$.
\end{remark}

\begin{remark}
The number of $k$-faces of an $m$-simplex $\sigma ^n$ in $\mathbb{R}^n$ ($n \geqslant m$) is equal to
the binomial coefficient $\binom{m+1}{k+1}$ (\emph{cf.} \cite[p.120]{Coxeter1948}).
For example, the number of 0-faces (vertices)
is $\binom{m+1}{1}= m+1$, the number of 1-faces (edges)
is $\binom{m+1}{2}= m(m+1)/2$, while
the number of $(m-1)$-faces
is $\binom{m+1}{m}= m+1$.
\end{remark}

\begin{definition}
The \emph{diameter}\/ of
an $m$-simplex $\sigma^m$ in $\mathbb{R}^n$ ($n \geqslant m$) denoted by ${\rm diam}(\sigma^m)$,
is the length of the longest edge (1-face) of $\sigma^m$
(\emph{cf.} \cite[p.607]{AlexandroffH1974}, \cite[p.812]{Whitehead1940}),
where the Euclidean norm, $\|\cdot\|$, is used here to measure distances.
The length of the shortest edge of $\sigma^m$ will be denoted by ${\rm shor}(\sigma^m)$.
The \emph{diameter}\/ of a bounded set $P$ in $\mathbb{R}^n$ is defined as
$ {\rm diam}(P) = \sup_{x,y \in P} \| x - y \|$, while
the \emph{width}\/ or \emph{breadth} of $P$ denoted by ${\rm wid}(P)$, is
the minimum distance of two parallel supporting hyperplanes (tac-hyperplanes) of $P$.
\end{definition}

\begin{definition}
Let $\sigma^m = [{\upsilon}^0, {\upsilon}^1, \ldots ,{\upsilon}^m ]$ be an $m$-simplex in $\mathbb{R}^n$ ($n \geqslant m$)
then, the \emph{barycenter}\/ or \emph{centroid}\/ of $\sigma^m$ and the \emph{barycenter}\/ or \emph{centroid}\/ of the
$i$-th $(m-1)$-face $\sigma^{m}_{\neg i}$, $i=0, 1, \ldots, m$, of $\sigma^m$
are respectively denoted by $\kappa^m$ and $\kappa_i^m$ and are given by
\begin{equation}\label{eq:Bar}
\kappa^m = \frac{1}{m+1} \sum_{j=0}^{m} {\upsilon}^j\,\,
\text{ \ and \ }\,\,
\kappa_i^m  = \frac{1}{m} \sum_{\substack{j=0\\j\ne i}}^m {\upsilon}^j,\,\,\,\, i=0, 1, \ldots, m.
\end{equation}
\end{definition}

\begin{remark}
By convexity the barycenter of any $m$-simplex $\sigma^m$ in $\mathbb{R}^n$
is a point in the relative interior of $\sigma ^m$.
\end{remark}

\begin{definition}
Let $\sigma^m = [{\upsilon}^0, {\upsilon}^1, \ldots ,{\upsilon}^m ]$ be an $m$-simplex in $\mathbb{R}^n$ ($n \geqslant m$)
then, $\sigma^m$ is called \emph{orthocentric}\/ if its $m+1$ altitudes intersect in a common point $o^m$, called its \emph{orthocenter}.
\end{definition}

\begin{remark}
The $m+1$ altitudes of an $m$-simplex are not necessarily concurrent
if $m \geqslant 3$.
\end{remark}

\begin{definition}
Let $\sigma^m = [{\upsilon}^0, {\upsilon}^1, \ldots ,{\upsilon}^m ]$ be an $m$-simplex in $\mathbb{R}^n$ ($n \geqslant m$)
and let $\kappa_i^m$ be the barycenter of the
$i$-th $(m-1)$-face $\sigma^{m}_{\neg i}$, opposite to vertex ${\upsilon}^i$ of $\sigma^m$
then, the 1-simplex ${\mu}_i^m =\bigl[{\upsilon}^i,\, \kappa_i^m\bigr]$ for $i=0, 1, \ldots, m$ is called the \emph{$i$-th median}\/ of~$\sigma^m$
that corresponds to the vertex ${\upsilon}^i$.
\end{definition}

\begin{notation}
The $(n-1)$-dimensional spherical surface is denoted by ${S}^{n-1}$ for which ${S}^{n-1} = \vartheta {B}^n_{r,c}$
where ${B}^n_{r,c} \subset \mathbb{R}^n$ denotes the standard closed Euclidean ball of radius $r > 0$ centered at a point $c$ in $\mathbb{R}^n$,
\emph{i.e.} ${B}^n_{r,c} = \{ x \in \mathbb{R}^n : \| x -c \| \leqslant r \}$.
\end{notation}

\begin{definition}
Assume that $\sigma^m$
is an $m$-simplex in\/ ${\mathbb{R}}^n$ $(n \geqslant m)$ and let
$\kappa^m$ be its barycenter. Then the smalest radius of the spherical surfaces ${S}^{n-1}$
that enclose $\sigma^m$ is called \emph{circumradius}\/ and it is denoted by $\rho_{\rm cir}^m$ and
the corresponding center denoted by $c_{\rm cir}^m$ is called \emph{circumcenter}.
If $c_{\rm cir}^m = \kappa^m$ then we call the corresponding circumradius as \emph{barycentric circumradius}\/ and we denote it as $\beta_{\rm cir}^m$.
\end{definition}

\begin{definition}
Let $\sigma^m$ be an $m$-simplex in\/ ${\mathbb{R}}^n$ $(n \geqslant m)$ with diameter ${\rm diam}(\sigma^m)$ and
barycenter $\kappa^m$.
Then,
the largest radius of the spherical surfaces ${S}^{n-1}$
that are enclosed within $\sigma^m$ is called \emph{inradius}\/ of $\sigma^m$, it is denoted by $\rho_{\rm inr}^m$ and
the corresponding center denoted by $c_{\rm inr}^m$ is called \emph{incenter}.
If $c_{\rm inr}^m = \kappa^m$ then the corresponding inradius denoted by $\beta_{\rm inr}^m$ is called \emph{barycentric inradius} and it
is given~by:
\begin{equation}\label{eq:BarInr}
\beta_{\rm inr}^m = \min_{x \in {\vartheta}{\sigma}^m} \|\kappa^{m} - x\|,
\end{equation}
where ${\vartheta}{\sigma}^m$ denotes the boundary of $\sigma^m$.
Furthermore, the following dimensionless quantity
\begin{equation}\label{eq:thic}
{\theta}({\sigma}^m) = \beta_{\rm inr}^m / {\rm diam}(\sigma^m),
\end{equation}
is called the \emph{thickness}\/ of $\sigma^m$
(\emph{cf.} \cite{Kojima1978,Munkres1966,Saigal1979,Whitehead1940}).
\end{definition}

Next, the straightforward generalization of the Apollonius theorem for $m$-simplices in\/ ${\mathbb{R}}^n$ $(n \geqslant m)$ follows.

\subsection{Generalization of the Apollonius theorem}\label{subsec:GenApo}

\paragraph*{\\[-0.3cm]
Apollonius' theorem
(\emph{cf.} Theorem~\ref{ApolloniusTh1}) can be formulated for 2-simplices as follows:}

\begin{theorem}[Apollonius' theorem for 2-simplices]\label{ApolloniusTh2}
Let $\sigma^2 = [{\upsilon}^0, {\upsilon}^1, {\upsilon}^2]$ be a $2$-simplex in $\mathbb{R}^n$ $(n \geqslant 2)$.
If $\mu_i^2$ is the $i$-th median of $\sigma^2$ that corresponds to the vertex ${\upsilon}^i$, $i=0,1,2$, it holds that:
\begin{equation}\label{eq:Apollonius2}
2\v j0i2  \nr {\upsilon}i{\upsilon}j^2 =
\sum_{\substack{p=0\\p\ne i}}^{1}
      \v q{p+1}i2  \nr {\upsilon}p{\upsilon}q^2 + 2^2\,\|\mu_i^2\|^2, \quad i=0, 1, 2.
\end{equation}
\end{theorem}

\begin{example}
 By setting in Eq.~(\ref{eq:Apollonius2}) $i=0$  it follows that:
 \[
 2 \bigl(\|{\upsilon}^0 - {\upsilon}^1\|^2 + \|{\upsilon}^0 - {\upsilon}^2\|^2 \bigr) = \|{\upsilon}^1 - {\upsilon}^2\|^2 + 2^2\,\|\mu_0^2\|^2,
 \]
(\emph{cf.} Theorem~\ref{ApolloniusTh1}).
\end{example}

\begin{theorem}[Generalization of the Apollonius theorem for simplices]\label{SimApolloniusTh}
Assume that $\sigma^m = [{\upsilon}^0, {\upsilon}^1, \ldots, {\upsilon}^m]$
is an $m$-simplex in\/ ${\mathbb{R}}^n$ $(n \geqslant m)$
with barycenter $\kappa^m$.
Assume further that $\kappa_i^m$ is the barycenter of the
$i$-th $(m-1)$-face $\sigma^{m}_{\neg i}$ of $\sigma^m$ and
let
$\mu_i^m = [\kappa_i,\, {\upsilon}^i]$ be the $i$-th median that corresponds to the vertex ${\upsilon}^i$ for\, $i=0, 1, \ldots, m$.
Then it holds that:
\begin{equation}\label{eq:2.0}
m\v j0im  \nr {\upsilon}i{\upsilon}j^2 =
\sum_{\substack{p=0\\p\ne i}}^{m-1}
      \v q{p+1}im  \nr {\upsilon}p{\upsilon}q^2 + m^2 \|\mu_i^m\|^2, \quad i=0,1, \ldots, m.
\end{equation}
\end{theorem}
\begin{proof}
Assume that ${\upsilon}^i=({\upsilon}^i_1,{\upsilon}^i_2,\ldots,{\upsilon}^i_n)$ is a vertex
 of $\sigma^m$, then we have
$\|\mu_i^m\|^2=\|\kappa_i-{\upsilon}^i\|^2$.
Equivalently, we can obtain
\begin{equation}\label{eq:2.1}
\|\mu_i^m\|^2={\frac{1}{m^2}}\, \mathit\Psi_i,
\end{equation}
where
\[
\mathit\Psi_i =  \sum^n_{t=1}\,\Biggl\{\, \v l0im
   {\upsilon}^l_t - m {\upsilon}^i_t\Biggr\}^2.
\]
The term $\mathit\Psi_i$ can be written as follows:
\begin{align*}
&\mathit\Psi_i = \sum^n_{t=1} \Biggl\{ \biggl[ \v l0im
   {\upsilon}^l_t  \biggr] ^2 + m^2\bigl[{\upsilon}^i_t\bigr]^2 -2 m {\upsilon}^i_t \v l0im {\upsilon}^l_t \Biggr\}  \Longleftrightarrow \\[0.2cm]
&\mathit\Psi_i =  \sum^n_{t=1} \Biggl\{ \v l0im
\bigl[ {\upsilon}^l_t \bigr]^2 + 2
\sum_{\substack{p=0\\p\ne i}}^{m-1}
     \v q{p+1}im {\upsilon}^p_t {\upsilon}^q_t  +
    m^2 \bigl[{\upsilon}^i_t\bigr]^2 -2 m {\upsilon}^i_t \v l0im {\upsilon}^l_t \Biggr\},
\end{align*}
from which we obtain
\begin{align*}
&\mathit\Psi_i =  \sum^n_{t=1} \Biggl\{m \v l0im
\bigl[ {\upsilon}^l_t \bigr]^2 +
    m^2 \bigl[{\upsilon}^i_t\bigr]^2 - 2 m {\upsilon}^i_t \v l0im {\upsilon}^l_t  - (m-1) \v l0im
\bigl[ {\upsilon}^l_t \bigr]^2  \\
      &  \kern1.85cm +
 2 \sum_{\substack{p=0\\ p\ne i}}^{m-1}
     \v q{p+1}im {\upsilon}^p_t {\upsilon}^q_t \Biggr\} \Longleftrightarrow \\[0.2cm]
&\mathit\Psi_i
    =  m\sum^n_{t=1} \Biggl\{ \v l0im
\bigl[ {\upsilon}^l_t \bigr]^2 +
    m \bigl[{\upsilon}^i_t\bigr]^2 - 2  {\upsilon}^i_t \v l0im {\upsilon}^l_t \Biggr\}\\
        & \kern1.85cm -\sum^n_{t=1} \Biggl\{ (m-1) \v l0im
\bigl[ {\upsilon}^l_t \bigr]^2  -
2 \sum_{\substack{p=0\\p\ne i}}^{m-1}
     \v q{p+1}im {\upsilon}^p_t {\upsilon}^q_t \Biggr\},
\end{align*}
or equivalently, after some algebraic manipulations, we obtain
\begin{equation}\label{eq:2.3}
\mathit\Psi_i= m\v j0im  \nr {\upsilon}i{\upsilon}j^2-
    \sum_{\substack{p=0\\p\ne i}}^{m-1}
      \v q{p+1}im  \nr {\upsilon}p{\upsilon}q^2.
\end{equation}
Companying Eq.~(\ref{eq:2.1}) and Eq.~(\ref{eq:2.3}) we have
\begin{equation}\label{eq:2.4}
\|\mu_i^m\|^2 = \frac{1}{m^2} \Biggl(m\v j0im  \nr {\upsilon}i{\upsilon}j^2-
    \sum_{\substack{p=0\\p\ne i}}^{m-1}
      \v q{p+1}im  \nr {\upsilon}p{\upsilon}q^2\Biggr).
\end{equation}
Thus the theorem is proved.
\end{proof}

\subsection{Generalizations of median properties}\label{subsec:MedProp}

\begin{remark}
The generalization of the Apollonius Theorem~\ref{SimApolloniusTh} provides the following immediate result
which gives an expression for the medians length of a simplex that can be calculated in terms of its edge lengths.
\end{remark}

\begin{corollary}[Medians length]\label{MedCor}
Assume that  $\sigma^m = [{\upsilon}^0, {\upsilon}^1, \ldots, {\upsilon}^m]$
is an $m$-simplex in\/ ${\mathbb{R}}^n$ $(n \geqslant m)$.
Then the lengths of the $m+1$ medians $\mu_i^m$ for $i=0, 1, \ldots, m$
of $\sigma^m$ are given as follows:
\begin{equation}\label{eq:MedCor}
\|\mu_i^m\| =  \frac{1}{m} \Biggl(m \v j0im  \nr {\upsilon}i{\upsilon}j^2
    -\sum_{\substack{p=0\\p\ne i}}^{m-1}
      \v q{p+1}im  \nr {\upsilon}p{\upsilon}q^2\Biggr)^{1/2} .
\end{equation}
\end{corollary}
\begin{proof}
Since the right hand side of Eq.~(\ref{eq:2.4}) is positive, the proof follows from Eq.~(\ref{eq:2.4}).
\end{proof}

\begin{remark}
The length of the $i$-th median $\mu_i^m$ for $i=0, 1, \ldots, m$ of $\sigma^m$ that corresponds to the vertex ${\upsilon}^i$ is given in terms of the lengths of
(a) the $m$ edges of $\sigma^m$ that concur in the vertex ${\upsilon}^i$
and
(b) the $\binom{m}{2}$ edges of the $i$-th $(m-1)$-face $\sigma^{m}_{\neg i}$ of $\sigma^m$.
\end{remark}

\begin{remark}
It is well known the property of a triangle that its three medians concur in the centroid
that divides each of them in the ratio $1:2$, where the longer segment being on the side of the vertex.
In 1565 Commandino\footnote{\,Federico Commandino (1509 -- 1575), Italian mathematician.}
in his work
\emph{``De centro gravitates solidorum''}\/ (The center of gravity of solids),
extended the above property of triangles for the case of the medians of a tetrahedron
(\emph{i.e.} the line segments that join each vertex of the tetrahedron to the barycenter
of the opposite face). The corresponding result is known as \emph{Commandino's theorem}~\cite[p.97]{AlsinaN2015}.
\end{remark}

\begin{theorem}[Commandino's theorem (1565)]\label{CommandinoTh}
The four medians of a tetrahedron concur in a
point that divides each of them in the ratio $1:3$, where the longer segment being on
the side of the vertex of the tetrahedron.
\end{theorem}

Next, the generalization of Commandino's theorem for $m$-simplices in\/ ${\mathbb{R}}^n$ $(n \geqslant m)$ follows.

\begin{theorem}[Generalization of Commandino's theorem for simplices]\label{SimCommandinoTh}
Assume that the hypotheses of Theorem~\ref{SimApolloniusTh} are fulfilled,
then the $m+1$ medians of
$\sigma^m$ concur in the barycenter $\kappa^m$ of $\sigma^m$
that divides each of them in the ratio $1:m$, where the longer segment being on
the side of the vertex of $\sigma^m$.
\end{theorem}
\begin{proof}
Let ${\upsilon}^i=({\upsilon}^i_1,{\upsilon}^i_2,\ldots,{\upsilon}^i_n)$ be a vertex of $\sigma^m$, then
\begin{align*}
\kern-2cm \kappa^m - {\upsilon}^i =
\Bigl\{&\Bigl[\frac{1}{m+1}\,\Bigl({\upsilon}^0_1 + {\upsilon}^1_1 + \cdots + {\upsilon}^m_1\Bigr)- {\upsilon}^i_1 \Bigr], \ldots ,\\
&\Bigl[\frac{1}{m+1}\,\Bigl({\upsilon}^0_n + {\upsilon}^1_n + \cdots + {\upsilon}^m_n\Bigr)- {\upsilon}^i_n \Bigr]\Bigr\} \Longleftrightarrow
\end{align*}
\begin{align}\label{eq:Km1}
\kappa^m - {\upsilon}^i =
\frac{1}{m+1}\,\Bigl\{&\Bigl({\upsilon}^0_1 + {\upsilon}^1_1 + \cdots + {\upsilon}^{i-1}_1 - m { \upsilon}^i_1 + {\upsilon}^{i+1}_1 + \cdots + {\upsilon}^m_1\Bigr), \ldots ,\\ \nonumber
&\,\Bigl({\upsilon}^0_n + {\upsilon}^1_n + \cdots + {\upsilon}^{i-1}_n - m { \upsilon}^i_n + {\upsilon}^{i+1}_n + \cdots + {\upsilon}^m_n\Bigr)\Bigr\}.
\end{align}
Similarly for the barycenter $\kappa^{m}_{i}$ of the oppositive to the vertex ${\upsilon}^i$ $(m-1)$-face $\sigma^{m}_{\neg i}$ of $\sigma^m$ we have
\begin{align*}
\kern0.5cm \kappa^{m}_{i} - {\upsilon}^i =
\Bigl\{&\Bigl[\frac{1}{m}\,\Bigl({\upsilon}^0_1 + {\upsilon}^1_1 + \cdots + {\upsilon}^{i-1}_1 + {\upsilon}^{i+1}_1 + \cdots + {\upsilon}^m_1\Bigr)- {\upsilon}^i_1 \Bigr], \ldots ,\\
&\Bigl[\frac{1}{m}\,\Bigl({\upsilon}^0_n + {\upsilon}^1_n + \cdots + {\upsilon}^{i-1}_n + {\upsilon}^{i+1}_n + \cdots + {\upsilon}^m_n \Bigr)- {\upsilon}^i_n \Bigr]\Bigr\}
\Longleftrightarrow
\end{align*}
\begin{align}\label{eq:Km2}
\kappa^{m}_{i} - {\upsilon}^i =
\frac{1}{m}\,\Bigl\{&\Bigl({\upsilon}^0_1 + {\upsilon}^1_1 + \cdots + {\upsilon}^{i-1}_1 - m {\upsilon}^i_1 + {\upsilon}^{i+1}_1 + \cdots + {\upsilon}^m_1\Bigr), \ldots ,\\ \nonumber
&\,\Bigl({\upsilon}^0_n + {\upsilon}^1_n + \cdots + {\upsilon}^{i-1}_n - m {\upsilon}^i_n + {\upsilon}^{i+1}_n + \cdots + {\upsilon}^m_n\Bigr)\Bigr\}.
\end{align}
Companying Eq.~(\ref{eq:Km1}) and Eq.~(\ref{eq:Km2}) we obtain
\begin{equation}\label{eq:Km3}
(m+1) (\kappa^{m} - {\upsilon}^i) = m (\kappa^{m}_{i} - {\upsilon}^i), \quad i=0, 1, \ldots, m.
\end{equation}
Thus, the points ${\upsilon }^i$, $\kappa^m$ and $\kappa^m_{i}$ are collinear points
and the $m+1$ medians ${\mu}_i^m$, $i=0, 1, \ldots, m$
of $\sigma^m$ concur in the barycenter $\kappa^m$ of $\sigma^m$.
Next, since
$\kappa^{m}_{i} - {\upsilon}^i = (\kappa^{m}_{i} - \kappa^{m}) + (\kappa^{m} - {\upsilon}^i)$,
using Eq.~(\ref{eq:Km3}) we obtain
\begin{equation}\label{eq:Km4}
\| \kappa^{m} - {\upsilon}^i \| = m \| \kappa^{m}_{i} - \kappa^{m} \|, \quad i=0, 1, \ldots, m.
\end{equation}
Therefore, the barycenter $\kappa^m$ of $\sigma^m$
divides each median ${\mu}_i^m$, $i=0, 1, \ldots, m$  in the ratio $1:m$, where the longer segment being on
the side of the vertex of $\sigma^m$.
Thus the theorem is proved.
\end{proof}

The Apollonius Theorem~\ref{ApolloniusTh1} gives the following immediate interesting results (\emph{cf.} \cite[p.68]{Johnson1960}).

\begin{theorem}[Median properties for triangles]\label{ApolloniusCorr}
In any triangle $\Delta A B C$ with vertices $A$, $B$, $C$ and lengths of its sides $\ell (\,\,\overline{\!\!AB})$, $\ell (\,\,\overline{\!\!AC})$ and
$\ell (\,\,\overline{\!\!BC})$,
if\,\,\, $\overline{\!\!AD}$,\,\, $\overline{\!\!BE}$ and\,\, $\overline{\!\!CF}$ are its medians with lengths, respectively,
$\ell (\,\,\overline{\!\!AD})$, $\ell (\,\,\overline{\!\!BE})$ and\, $\ell (\,\,\overline{\!\!CF})$ and if\, $M$
is the median point of intersection of the medians, then the following hold:
\begin{align}
\ell (\,\,\overline{\!\!AD})^{2} + \ell (\,\,\overline{\!\!BE})^{2} + \ell (\,\,\overline{\!\!CF})^{2} &= \frac{3}{4} \bigl[\ell (\,\,\overline{\!\!AB})^{2} +
\ell (\,\,\overline{\!\!AC})^{2} + \ell (\,\,\overline{\!\!BC})^{2}\bigr], \label{eq:ApolloniusCorr1}\\[0.15cm]
\ell (\,\,\overline{\!\!MA})^{2} + \ell (\,\,\overline{\!\!MB})^{2} + \ell (\,\,\overline{\!\!MC})^{2} &= \frac{1}{3} \bigl[\ell (\,\,\overline{\!\!AB})^{2} +
\ell (\,\,\overline{\!\!AC})^{2} + \ell (\,\,\overline{\!\!BC})^{2}\bigr].\label{eq:ApolloniusCorr2}
\end{align}
\end{theorem}

Straightforward generalizations of the above results are given as follows.

\begin{theorem}[Median properties for simplices]\label{corr:corrgen}
Let $\sigma^m = [{\upsilon}^0, {\upsilon}^1, \ldots, {\upsilon}^m]$
be an $m$-simplex in\/ ${\mathbb{R}}^n$ $(n \geqslant m)$ with barycenter
$\kappa^m$.
Assume that $\kappa_i^m$ is the barycenter of the
$i$-th $(m-1)$-face $\sigma^{m}_{\neg i}$ of $\sigma^m$ and
let
$\mu_i^m = [{\upsilon}^i,\, \kappa_i^m]$ be the $i$-th median that corresponds to the vertex ${\upsilon}^i$ for $i=0, 1, \ldots, m$.
Then the following hold:
\begin{align}
\sum_{i=0}^{m} \| \mu_i^m \|^2 &= \frac{m+1}{m^2} \sum_{p=0}^{m-1} \sum_{q=p+1}^{m}\| {\upsilon}^p - {\upsilon}^q \|^2, \label{eq:corrgen1}\\[0.15cm]
\sum_{i=0}^{m} \| \kappa^{m} - {\upsilon}^i \|^2 &= \frac{1}{m+1} \sum_{p=0}^{m-1} \sum_{q=p+1}^{m}\| {\upsilon}^p - {\upsilon}^q \|^2. \label{eq:corrgen2}
\end{align}
\end{theorem}
\begin{proof}
The proof of Eq.~(\ref{eq:corrgen1}) can be obtained using Eq.~(\ref{eq:2.4}) for $i=0, 1, \ldots, m$ and
by adding the corresponding expressions to obtain the term $\sum_{i=0}^{m} \| \mu_i^m \|^2$.
Since an edge $[{\upsilon}^i, {\upsilon}^j]$ of $\sigma^m$ corresponds to the two vertices ${\upsilon}^i$ and ${\upsilon}^j$ of $\sigma^m$
then for the first term of the right hand side of Eq.~(\ref{eq:2.4}) we have
\begin{equation}\label{eq:corrgen3}
\sum_{i=0}^{m} \v j0im  \nr {\upsilon}i{\upsilon}j^2 = 2 \sum_{p=0}^{m-1} \sum_{q=p+1}^{m}\| {\upsilon}^p - {\upsilon}^q \|^2,
\end{equation}
since each edge is added twice which results to twice the sum of the squares of the lengths of all the edges of $\sigma^m$.

Also, since an edge $[{\upsilon}^i, {\upsilon}^j]$ of $\sigma^m$ can constitute an edge
of the $k$-th $(m-1)$-face $\sigma^{m}_{\neg k}$ of $\sigma^m$ for the $m-1$ vertices ${\upsilon}^k$ of $\sigma^m$
where $k$ is different from $i$ and $j$, then for the second term of the right hand side of Eq.~(\ref{eq:2.4}) we have
\begin{equation}\label{eq:corrgen4}
 -\sum_{i=0}^{m}
    \sum_{\substack{p=0\\p\ne i}}^{m-1}
      \v q{p+1}im  \nr {\upsilon}p{\upsilon}q^2 = -(m-1) \sum_{p=0}^{m-1} \sum_{q=p+1}^{m}\| {\upsilon}^p - {\upsilon}^q \|^2,
\end{equation}
since each edge is added $m-1$ times which results to $m-1$ times the sum of the squares of the lengths of all the edges of $\sigma^m$.

Companying Eq.~(\ref{eq:2.4}), Eq.~(\ref{eq:corrgen3}) and Eq.~(\ref{eq:corrgen4}) we obtain
\[
\sum_{i=0}^{m} \| \mu_i^m \|^2 = \frac{1}{m^2} \Biggl(2 m \sum_{p=0}^{m-1} \sum_{q=p+1}^{m}\| {\upsilon}^p - {\upsilon}^q \|^2-
    (m-1) \sum_{p=0}^{m-1} \sum_{q=p+1}^{m}\| {\upsilon}^p - {\upsilon}^q \|^2 \Biggr),
\]
which proves Eq.~(\ref{eq:corrgen1}).
Next, using Eq.~(\ref{eq:Km3}), we have
\begin{equation}\label{eq:corrgen5}
\| \mu_i^m \|^2 = \|\kappa^{m}_{i} - {\upsilon}^i \|^2 = \frac{(m+1)^2}{m^2} \| \kappa^{m} - {\upsilon}^i \|^2, \quad \mbox{ for } \quad i=0, 1, \ldots, m.
\end{equation}
Companying the above relations and Eq.~(\ref{eq:corrgen1}) we obtain Eq.~(\ref{eq:corrgen2}).
Thus the theorem is proved.
\end{proof}

\begin{remark}\label{rem:Fermat}
We note in passing that, Eq.~(\ref{eq:corrgen2}) is related to
the well known \emph{Fermat's problem}~\cite{Coxeter1969,GueronT2002,Johnson1960}.
Specifically, in 1643
Fermat\footnote{\,Pierre de Fermat (1607 -- 1665), French mathematician.},
motivated by a letter of
Descartes\footnote{\,Ren\'{e} Descartes (1596 -- 1650), French philosopher and mathematician.},
proposed the following problem:
``\emph{Find a point $F$ for which the sum of the Euclidean distances from $F$ to the vertices $A$, $B$ and $C$
of a given triangle $\Delta A B C$ is minimized.}''
This problem was first raised by Fermat in a private letter to
Torricelli\footnote{\,Evangelista Torricelli (1608 -- 1647), Italian physicist and mathematician.}
who solved it.
Torricelli's solution was published in 1659 by his student
Viviani\footnote{\,Vincenzo Viviani (1622 -- 1703), Italian mathematician.}.
Therefor, \emph{Fermat's point}\/ $F$ of a triangle is also called \emph{Torricelli point}\/ or \emph{Fermat-Torricelli point}.
Notice that, the Fermat point $F$ of an equilateral (regular) triangle coincides with the
barycenter (centroid) $\kappa^2$ of the triangle~(\emph{cf.} Eq.~(\ref{eq:corrgen2})).
Also, it can be easily proved that in the case of regular simplices $\sigma^m$ in\/ ${\mathbb{R}}^n$
with diameter ${\rm diam}(\sigma^m)$ it holds that
$\sum_{i=0}^{m} \| \kappa^{m} - {\upsilon}^i \| =\sqrt{m(m+1)/2}\,\, {\rm diam}(\sigma^m)$ (\emph{cf}.\ Corollary~\ref{cor:reg-Fermat}).
It is worth noting that, in 1977 Alexander \cite{Alexander1977} has shown that
only the regular simplex has minimal diameter among all simplices inscribed in a sphere of ${\mathbb{R}}^n$ and containing
the center of the sphere.
\end{remark}

\begin{remark}\label{rem:reg-prop}
It is worth mentioning that, the regular simplex besides its symmetry exhibits interesting and useful properties
including, among others, the maximizing properties of the regular simplex
that have been given by Tanner in 1974~\cite{Tanner1974}.
Specifically, Tanner has shown
that the regular simplex maximizes the sum of the squared contents of all
$m$-faces, for all $m$, when the sum of squared edge (line) lengths is fixed.
Thus, the regular simplex has the largest total length of all
edges (joining lines), total area of all 2-faces (triangles), total volume of all 3-faces (tetrahedra),
\emph{etc.}, for a fixed sum of squared edge lengths~\cite{Tanner1974}.
In addition,
in 1977 Alexander~\cite{Alexander1977} proved the conjecture of
Sallee\footnote{\,George Thomas Sallee (1940 -- 2019), American mathematician.}
which states that:
\emph{``Only the regular simplex has maximal width among all simplices inscribed in a sphere of ${\mathbb{R}}^n$.''}
Note that, the width of the regular simplex in ${\mathbb{R}}^n$ with diameter 1 is approximately $\sqrt{2 / n}$
and specifically $\sqrt{2 / (n +1)}$ if $n$ is odd and $\sqrt{2 (n + 1)}/\sqrt{n (n + 2)}$ if $n$ is even
(\emph{cf}.\ \cite{Alexander1977}, \cite{GritzmannK1992}, \cite{Har-PeledR2023}),
for a short proof, see Theorem \ref{thm:reg-width}.
In addition,
regular simplices are orthocentric (\emph{cf}.\ Theorem~\ref{thm:correg}).
\end{remark}

\begin{theorem}[Medians of regular simplices]\label{thm:correg}
Let $\sigma^m = [{\upsilon}^0, {\upsilon}^1, \ldots, {\upsilon}^m]$
be a regular $m$-simplex in\/ ${\mathbb{R}}^n$ $(n \geqslant m)$
with edge length ${\rm diam}(\sigma^m)$ and barycenter $\kappa^m$.
Let $\kappa_i^m$ be the barycenter of the
$i$-th $(m-1)$-face $\sigma^{m}_{\neg i}$ of $\sigma^m$ and
let $\mu_i^m = [{\upsilon}^i,\, \kappa_i^m]$ be the $i$-th median that corresponds to the vertex ${\upsilon}^i$ for $i=0, 1, \ldots, m$.
Then, a) the $i$-th median of $\sigma^m$ coincides with its corresponding $i$-th altitude,
b) $\sigma^m$ is orthocentric and c) the orthocenter $o^m$ of $\sigma^m$ coincides with its barycenter $\kappa^m$.
\end{theorem}
\begin{proof}
By Eq.~(\ref{eq:2.4}) for the $i$-th median we obtain
\[
\|{\upsilon}^i -\kappa_i^m\|^2 = \frac{1}{m^2} \left( m^2 - \binom{m}{2} \right) {\rm diam}(\sigma^m)^2, \quad i=0, 1, \ldots, m,
\]
or equivalently we have
\begin{equation}\label{eq:correg1}
\|{\upsilon}^i -\kappa_i^m\|^2 = \frac{m+1}{2m}\,{\rm diam}(\sigma^m)^2, \quad i=0, 1, \ldots, m.
\end{equation}
Let $\mu_{ij}^m = [{\upsilon}^j,\, \kappa_{ij}^m]$ be the $j$-th median, $j=0, 1, \ldots, m$, $j\neq i$, that corresponds to the vertex ${\upsilon}^j$
 of the $i$-th $(m-1)$-face $\sigma^{m}_{\neg i}$.
Then, similarly to the above we have
\begin{equation}\label{eq:correg2}
\|{\upsilon}^j -\kappa_{ij}^m\|^2
= \frac{m}{2(m-1)}\,{\rm diam}(\sigma^m)^2.
\end{equation}
Using Eq.~(\ref{eq:Km3}) we obtain
\begin{equation}\label{eq:correg23}
\|{\upsilon}^j -\kappa_{i}^m\|^2 = \frac{(m-1)^2}{m^2}\|{\upsilon}^j -\kappa_{ij}^m\|^2.
\end{equation}
Companying Eqs.~(\ref{eq:correg2}) and (\ref{eq:correg23}) we have
\begin{equation}\label{eq:correg3}
\|{\upsilon}^j -\kappa_{i}^m\|^2 = \frac{m-1}{2m}\,{\rm diam}(\sigma^m)^2, \quad j=0, 1, \ldots, m, \kern0.3cm j\neq i.
\end{equation}
Companying Eqs.~(\ref{eq:correg1}) and (\ref{eq:correg3}) we obtain
\[
\|{\upsilon}^i -\kappa_i^m\|^2 + \|{\upsilon}^j -\kappa_{i}^m\|^2 = {\rm diam}(\sigma^m)^2.
\]
Since $\|{\upsilon}^i -{\upsilon}^j\|^2 = {\rm diam}(\sigma^m)^2$ we have
\begin{equation}\label{eq:GenPyth}
\|{\upsilon}^i -\kappa_i^m\|^2 + \|{\upsilon}^j -\kappa_{i}^m\|^2 = \|{\upsilon}^i -{\upsilon}^j\|^2.
\end{equation}
Thus, by the Pythagorean theorem for right triangles for the 2-simplex $[{\upsilon}^i, \kappa_i^m, {\upsilon}^j]$
we obtain that the $i$-th median $\mu_i^m = [{\upsilon}^i, \kappa_i^m]$ of $\sigma^m$ is also its
$i$-th altitude.
Also, due to Theorem~\ref{SimCommandinoTh}, $\sigma^m$ is orthocentric and its orthocenter $o^m$ coincides with its barycenter $\kappa^m$.
Thus the theorem is proved.
\end{proof}

\begin{remark}
Allow us to point out that, the Pythagorean theorem for right triangles due to
Pythagoras\footnote{\,Pythagoras of Samos (\emph{c}.\ 570 -- \emph{c}.\ 495 B.C.), Greek mathematician and philosopher.}
states that: \emph{``The sum of the squares on the legs of a right triangle is equal to the square on the hypotenuse (the side opposite the right angle).''}
This theorem is included by
Euclid\footnote{\,Euclid (\emph{c}.\ 323 -- \emph{c}.\ 285 B.C.), Greek mathematician, considered as the ``Father of Geometry''.}
in his seminal work \emph{``Elements''}\/ and specifically
in Book I, Proposition~47~(\emph{cf.}~\cite{Euclid1482}).
\end{remark}

\begin{remark}
Eq.~(\ref{eq:GenPyth}) could be named, the \emph{generalized Pythagorean equation for regular simplices}.
\end{remark}

Based on the above, in conclusion, we point out the following theorem.

\begin{theorem}[Generalization of Pythagoras' theorem for regular simplices]\label{thm:GenPyth}
Let $\sigma^m = [{\upsilon}^0, {\upsilon}^1, \ldots, {\upsilon}^m]$
be a regular $m$-simplex in\/ ${\mathbb{R}}^n$ $(n \geqslant m)$
and let $\kappa_i^m$  for $i=0, 1, \ldots, m$ be the barycenter of the
$i$-th $(m-1)$-face $\sigma^{m}_{\neg i}$ opposite to vertex ${\upsilon}^i$ of $\sigma^m$. Then it holds that
\begin{equation}\label{eqthm:GenPyth}
\|{\upsilon}^i -\kappa_i^m\|^2 + \|{\upsilon}^j -\kappa_{i}^m\|^2 = \|{\upsilon}^i -{\upsilon}^j\|^2, \quad i,j=0, 1, \ldots, m, \kern0.3cm j\neq i.
\end{equation}
\end{theorem}
\begin{proof}
The proof follows from Eq.~(\ref{eq:GenPyth}).
\end{proof}

\begin{remark}
In 1803
Carnot\footnote{\,Lazare Nicolas Marguerite Carnot (1753 -- 1823), French mathematician, physicist and politician.}
\cite{Carnot1803}
gave a theorem known as \emph{Carnot's theorem}\/ which states that:
\emph{``The sum of the distances from the circumcenter to the three sides of a
triangle is equal to the radius (circumradius) of the circumscribed circle plus the radius (inradius) of
the inscribed circle''}~\cite{Perrier2007}.
The most known application of Carnot's theorem is in the proof of
the \emph{Japanese theorem for a cyclic polygon}\/
(\emph{i.e.} polygon with vertices upon which a circle can be circumscribed)
which states that: \emph{``In any triangulation of a cyclic polygon, the sum of inradii of triangles is constant}''.
We note in passing that, according to
Hayashi\footnote{\,Tsuruichi Hayashi (1873 -- 1935), Japanese mathematician and historian.}:
``It was the ancient custom of
Japanese mathematicians to inscribe their discoveries on
tablets which were hung in the temples, to the glory of the
gods and the honor of the authors''.
The {Japanese theorem for cyclic polygons} has been exhibited in 1800~\cite[p.193]{Johnson1960} and thus it is also known as \emph{Japanese temple theorem}.
\end{remark}

Based on the above we give the following generalization of Carnot's theorem
for regular simplices.

\begin{theorem}[Generalization of Carnot's theorem for regular simplices]\label{thm:GenCarnot}
Let $\sigma^m = [{\upsilon}^0, {\upsilon}^1, \ldots, {\upsilon}^m]$
be a regular $m$-simplex in\/ ${\mathbb{R}}^n$ $(n \geqslant m)$
with barycenter $\kappa^m$
and let $\kappa_i^m$  for $i=0, 1, \ldots, m$ be the barycenter of the
$i$-th $(m-1)$-face $\sigma^{m}_{\neg i}$ opposite to vertex ${\upsilon}^i$ of $\sigma^m$.
Then\\
(a) the distance from the barycenter to the $j$-th $(m-1)$-face $\sigma^{m}_{\neg j}$ is given by
\begin{equation}\label{eqthm:GenCarnot1}
\min_{x \in \sigma^{m}_{\neg j}} \|\kappa^{m} - x\| = \|\kappa^m - \kappa_{j}^m\|, \quad \forall\, j=0, 1, \ldots, m,
\end{equation}
(b) the
length of the barycentric circumradius $\beta_{\rm cir}^m$ of $\sigma^m$ is given by
\begin{equation}\label{eqthm:GenCarnot2}
\beta_{\rm cir}^m = \|{\upsilon}^j - \kappa^m\|, \quad \forall\, j=0, 1, \ldots, m,
\end{equation}
(c) the
length of the barycentric inradius $\beta_{\rm inr}^m$ of $\sigma^m$ is given by
\begin{equation}\label{eqthm:GenCarnot3}
\beta_{\rm inr}^m = \|\kappa^m - \kappa_{j}^m\|, \quad \forall\, j=0, 1, \ldots, m,
\end{equation}
and (d) it holds that \vspace*{-0.2cm}
\begin{equation}\label{eqthm:GenCarnot4}
\sum_{i=0}^m \|\kappa^m - \kappa_i^m\| = \beta_{\rm cir}^m + \beta_{\rm inr}^m.
\end{equation}
\end{theorem}
\begin{proof}
It is evident that in a regular simplex all medians are of equal length (\emph{cf.} Eq.~(\ref{eq:correg1}))
and in general by Theorem~\ref{SimCommandinoTh} concur in the barycenter that divides each of them in the ratio $1:m$.
Thus the proof can be obtained by Theorem~\ref{thm:correg}.
\end{proof}

\begin{remark}
To clarify the notations in Theorem \ref{thm:GenCarnot}, we note that in the case of regular $m$-simplices
the barycenter $\kappa^m$ is the same with the circumcenter $c_{\rm cir}^m$ and the incenter $c_{\rm inr}^m$.
Also, the barycentric circumradius $\beta_{\rm cir}^m$ is the same with the circumradius $\rho_{\rm cir}^m$
while the barycentric inradius $\beta_{\rm inr}^m$ is the same with the inradius~$\rho_{\rm inr}^m$.
\end{remark}

\section{Applications of the generalized Apollonius theorem}\label{sec:Appl}
\subsection{Simplex and set enclosing}\label{subsec:SimCov}

\begin{remark}
The generalization of the Apollonius theorem for simplices can be applied for obtaining
an estimation of the radius of a spherical surface that encloses a given simplex or
a bounded subset of $\mathbb{R}^n$.
\end{remark}

\begin{lemma}\label{lem:circ}
Assume that $\sigma^m = [{\upsilon}^0, {\upsilon}^1, \ldots, {\upsilon}^m]$
is an $m$-simplex in\/ ${\mathbb{R}}^n$ $(n \geqslant m)$ and let
$\kappa^m$ be its barycenter. Then for each $i=0, 1, \ldots, m$ it holds that:
\begin{equation}\label{eq:3.1}
\|\kappa^{m} - {\upsilon}^i\| = \frac{1}{m + 1} \Biggl(m\v j0im  \nr {\upsilon}i{\upsilon}j^2-
    \sum_{\substack{p=0\\p\ne i}}^{m-1}
      \v q{p+1}im  \nr {\upsilon}p{\upsilon}q^2\Biggr)^{1/2}.
\end{equation}
\end{lemma}
\begin{proof}
The proof follows by companying Eq.~(\ref{eq:2.4}) and Eq.~(\ref{eq:Km3}).
\end{proof}

\begin{theorem}[Simplex barycentric enclosing]\label{thm:simcov}
Let $\sigma^m = [{\upsilon}^0, {\upsilon}^1, \ldots, {\upsilon}^m]$
be an $m$-simplex in\/ ${\mathbb{R}}^n$ $(n \geqslant m)$ and let
$\kappa^m$ be its barycenter. Then the barycentric circumradius $\beta_{\rm cir}^m$ is unique, it is given as follows:
\begin{equation}\label{eq:3.3}
\beta_{\rm cir}^m = \frac{1}{m + 1} \max_{0 \leqslant i \leqslant m} \Biggl\{m\v j0im  \nr {\upsilon}i{\upsilon}j^2-
    \sum_{\substack{p=0\\p\ne i}}^{m-1}
      \v q{p+1}im  \nr {\upsilon}p{\upsilon}q^2\Biggr\}^{1/2},
\end{equation}
and the spherical surface with center at $\kappa^m$ and radius $\beta_{\rm cir}^m$ encloses $\sigma^m$.
\end{theorem}
\begin{proof}
The proof is obvious from Lemma~\ref{lem:circ}.
\end{proof}

\begin{remark}
In 1901 Jung\footnote{\,Heinrich Wilhelm Ewald Jung (1876 -- 1953), German mathematician.}
\cite{Jung1901}
was first answered to the question of best possible estimate of the
radius of the smallest spherical surface enclosing
a bounded subset $P$ of $\mathbb{R}^n$ of a given diameter,
that is the maximal distance of any two points of $P$.
Particularly, Jung established results for the case of
finite point sets and indicated their extension to infinite sets.
In addition, in 1910 Jung~\cite{Jung1910} proposed necessary conditions on
the smallest circle enclosing a finite point set in a plane
(for a discussion on these issues and for a complete and elegant proof of Jung's theorem, we refer the interested reader to~\cite{BlumenthalW1941}).
\end{remark}

\begin{theorem}[Jung's enclosing theorem (1901)~\cite{Jung1901}]\label{JungTh}
Assume that ${\rm diam}(P)$ is the diameter of a bounded subset
${P}$ of\, $\mathbb{R}^n$ (containing more than a single point).
Then,
(a)~there exists a unique spherical surface of
circumradius $\rho_{\rm cir}^n$ enclosing $P$, and
(b)~$\rho_{\rm cir}^n \leqslant [n/(2n+2)]^{1/2}\, {\rm diam}(P)$.
\end{theorem}
\begin{proof}
This theorem has received quite a few proofs due to
S\"{u}ss (1936)~\cite{Suss1936}, Blumenthal and Wahlin (1941)~\cite{BlumenthalW1941},
Eggleston (1958)~\cite[Th.49, p.111]{Eggleston1958} and Guggenheimer (1977)~\cite[Pr.13-6, p.140]{Guggenheimer1977}, among others.
\end{proof}

In 1953 Gale~\cite{Gale1953} has provided the following  equivalent formulation of Jung's theorem.

\begin{theorem}[Gale's enclosing theorem (1953)~\cite{Gale1953}]\label{GaleTh-cir}
The circumscribed $n$-sphere of a regular $n$-simplex of diameter~1 will cover any $n$-dimensional set of diameter~1.
\end{theorem}

\begin{lemma}[Regular simplex enclosing]\label{lem:circ-reg}
Let $\sigma ^m$ be a regular $m$-simplex in ${\mathbb{R}}^n$ with edge length
${\rm diam}(\sigma^m)$,
then the barycentric circumradius $\beta_{\rm cir}^m$ is given by:
\begin{equation}\label{eq:equi}
\beta_{\rm cir}^m = [m/(2m+2)]^{1/2}\, {\rm diam}(\sigma^m).
\end{equation}
\end{lemma}
\begin{proof} Since the lengths of the $m$ edges that have the common vertex ${\upsilon}^i$
as well as the lengths of the $\binom{m}{2}$ edges of the $i$-th $(m-1)$-face $\sigma^{m}_{\neg i}$ are equal to ${\rm diam}(\sigma^m)$, by applying Lemma~\ref{lem:circ} we obtain
\begin{align*}
\beta_{\rm cir}^m = &\frac{1}{m+1}\left[m^2 {\rm diam}(\sigma^m)^2 -\binom{m}{2} {\rm diam}(\sigma^m)^2 \right]^{1/2}\\
                  = &\frac{1}{m+1}\left[ \frac{m (m+1)}{2}\right]^{1/2}{\rm diam}(\sigma^m)
                  = \left[\frac{m}{2(m+1)}\right]^{1/2} {\rm diam}(\sigma^m).
\end{align*}
Thus the lemma is proved.
\end{proof}

The following corollary is related to Fermat's problem (\emph{cf}.\ Remark~\ref{rem:Fermat}).

\begin{corollary}\label{cor:reg-Fermat}
Assume that $\sigma^m = [{\upsilon}^0, {\upsilon}^1, \ldots, {\upsilon}^m]$
is an $m$-simplex in\/ ${\mathbb{R}}^n$ $(n \geqslant m)$ with diameter ${\rm diam}(\sigma^m)$ and let
$\kappa^m$ be its barycenter.
Then it holds that
\begin{equation}\label{eq:reg-Fermat}
\sum_{i=0}^{m} \| \kappa^{m} - {\upsilon}^i \| =[{m (m + 1)/2}]^{1/2}\,\, {\rm diam}(\sigma^m).
\end{equation}
\end{corollary}
\begin{proof}
The proof follows from Eq.~(\ref{eq:equi}).
\end{proof}

\begin{remark}
In the case of regular simplices the barycentric circumradius and Jung's estimate are identical.
\end{remark}

\begin{remark}\label{rem:bar}
The barycentric circumradius of an $m$-simplex $\sigma^m$ in $\mathbb{R}^n$ with
diameter ${\rm diam}(\sigma^m)$ is not always greater than the
barycentric circumradius of any
regular $m$-simplex $\tau^m$ in $\mathbb{R}^n$ with the same diameter.
Consider, for example,
a regular $2$-simplex $\tau^2 = [{\upsilon}^0, {\upsilon}^1, {\upsilon}^2]$ with diameter 1
and let $\kappa^2$ be its barycenter.
By applying Lemma~\ref{lem:circ} we obtain $\|\kappa^2 -{\upsilon}^i\|= 3^{-1/2}$ for $i=0,1,2.$
Thus the barycentric circumradius of $\tau^2$ is $3^{-1/2}$.
Consider now the $2$-simplex $\sigma^2 = [\kappa^2, {\upsilon}^1, {\upsilon}^2]$.
The lengths of the edges of $\sigma^2$ are $3^{-1/2}$, $3^{-1/2}$ and 1. Therefore the diameter of $\sigma^2$ is 1.
By applying Theorem~\ref{thm:simcov} we find that the
barycentric circumradius of $\sigma^2$ is equal to
$[7/27]^{1/2}$, which is smaller than the estimation of the upper bound $3^{-1/2}$ for
the circumradius $\rho_{\rm cir}^2$ that is given by
the Jung enclosing Theorem~\ref{JungTh}.
\end{remark}

Based on the above we provide the following result.

\begin{theorem}[Simplex enclosing]\label{thm:simcovII}
Assume that $\sigma^m = [{\upsilon}^0, {\upsilon}^1, \ldots, {\upsilon}^m]$
is an $m$-simplex in\/ ${\mathbb{R}}^n$ $(n \geqslant m)$ with
diameter ${\rm diam}(\sigma^m)$ and let $\beta_{\rm cir}^m$ be its barycentric circumradius.
Then, there exists a unique spherical surface of
circumradius $\rho_{\rm cir}^m$ enclosing $\sigma^m$, and
$\rho_{\rm cir}^m \leqslant \min \bigl\{\beta_{\rm cir}^m,\, [m/(2m+2)]^{1/2}\, {\rm diam}(\sigma^m) \bigr\}$.
\end{theorem}
\begin{proof}
The proof is obvious from Theorem~\ref{thm:simcov} and Theorem~\ref{JungTh}.
\end{proof}

\begin{remark}
The computation of the barycentric circumradius does not require
any additional computational burden since it can be easily obtained
during the computation of the longest edge (diameter) of the simplex.
\end{remark}

Next, the important theorems due to
Carath\'{e}odory\footnote{\,Constantin Carath\'{e}odory (1873 -- 1950), Greek mathematician.},
Helly\footnote{\,Eduard Helly (1884 -- 1943), Austrian mathematician.},
Radon\footnote{\,Johann Karl August Radon (1887 -- 1956), Austrian mathematician.},
and
Tverberg\footnote{\,Helge Arnulf Tverberg (1935 -- 2020), Norwegian mathematician.}
are presented.

\begin{theorem}[Carath\'{e}odory's theorem (1907)~\cite{Caratheodory1907}]\label{CaratheodoryTh}
Any point in the convex hull of a finite point set in $\mathbb{R}^n$ is a convex combination
of some at most $n + 1$ of these points.
\end{theorem}

\begin{theorem}[Helly's theorem (1913)~\cite{Helly1923}]\label{HellyTh}
Let $C_1, C_2, \ldots, C_k$ be a finite family of convex subsets of\/ $\mathbb{R}^n$, with $k \geqslant n + 1$.
If the intersection of every $n + 1$ of these sets is nonempty, then the whole family has a nonempty intersection.
\end{theorem}

\begin{theorem}[Radon's theorem (1921)~\cite{Radon1921}]\label{RadonTh}
Any set of $n + 2$ points in $\mathbb{R}^n$ can be partitioned into two sets whose convex hulls intersect.
\end{theorem}

\begin{theorem}[Tverberg's theorem (1966)~\cite{Tverberg1966}]\label{TverbergTh}
Every set with at least $(p - 1)(n + 1) + 1$ points in\/ $\mathbb{R}^n$ can be
partitioned into $p$ subsets whose convex hulls all have at least one point in common.
\end{theorem}

\begin{remark}
In 1970
Rockafellar\footnote{\,Ralph Tyrrell Rockafellar (b.\ 1935), American mathematician.}
has pointed out that
\emph{``Carath\'{e}odory's theorem is the fundamental dimensionality result in
convexity theory, and it is the source of many other results in which
dimensionality is prominent''}~\cite{Rockafellar1970}.
In general, the convex hull of a subset $P$ of $\mathbb{R}^n$ can be obtained by carrying out all convex combinations of elements of $P$.
On the other hand, due to Carath\'{e}odory's theorem
it is not necessary to perform combinations involving
more than $n + 1$ elements at a time.
\end{remark}

\begin{remark}
Although Theorem \ref{HellyTh} was proposed by Helly in 1913, yet it was published in 1923~\cite{Helly1923}.
Radon's theorem gives the partition of finite point sets
so that the convex hulls of the parts intersect.
This partition is a special case of
the  \emph{Tverberg's partition}.
Tverberg gave the first proof of Theorem~\ref{TverbergTh} in 1966,
while in 1981 he gave a simpler proof~\cite{Tverberg1981}.
Obviously, for $p=2$ we obtain Radon's theorem,
while the case $n=2$ was proved by
Birch\footnote{\,Bryan John Birch (b. 1931), English mathematician.}
in 1959~\cite{Birch1959}.
\end{remark}

\begin{remark}
Theorems~\ref{CaratheodoryTh}, \ref{HellyTh} and \ref{TverbergTh}
are equivalent in the sense that each one can be deduced from another~\cite{DeLoeraGMM2019}.
\end{remark}

Based on the above Blumenthal and Wahlin in 1941~\cite{BlumenthalW1941}
have provided the following result that permits the reduction of the enclosing problem to a finite one concerning $n+1$ points.

\begin{theorem}[Blumenthal-Wahlin theorem (1941)~\cite{BlumenthalW1941}]\label{lem:Car}
If each set of\/ $n+1$ points of a subset $P$ of\/ $\mathbb{R}^n$ is enclosable by a spherical surface of a given radius,
then $P$ is itself enclosable by this spherical surface.
\end{theorem}

\begin{remark}
The barycentric circumradius $\beta_{\rm cir}^n=\beta_{\rm cir}^n(P)$ of a subset $P$ of $\mathbb{R}^n$ is the supremum of the barycentric circumradii of simplices
with vertices in $P$.
In the following theorem, an application of the above theorems has shown that a bounded set $P$ can be covered by a spherical surface
of circumradius~$\beta_{\rm cir}^n(P)$,
which in many cases gives a better result than Jung's enclosing theorem (\emph{cf.} Remark~\ref{rem:bar}).
\end{remark}

\begin{theorem}[Variant of Jung's enclosing theorem (1988)~\cite{Vrahatis1988}]\label{VrahatisTheoremScov}
Assume that ${\rm diam}(P)$ is the diameter of a bounded subset
${P}$ of\, $\mathbb{R}^n$ (containing more than a single point)
and let $\beta_{\rm cir}^n=\beta_{\rm cir}^n(P)$ be its barycentric circumradius.
Then,
(a)~there exists a unique spherical surface of
circumradius $\rho_{\rm cir}^n$ enclosing $P$, and
(b)~$\rho_{\rm cir}^n \leqslant \min \bigl\{\beta_{\rm cir}^n(P),\, [n/(2n+2)]^{1/2}\, {\rm diam}(P) \bigr\}$.
\end{theorem}

\subsubsection{Jung's enclosing theorem and the minimum enclosing ball problem}

The enclosing (covering) problem is related to the well known and widely applied
\emph{minimum enclosing ball}\/ (MEB) problem,
which refers to the \emph{determination of the unique
spherical surface of smallest radius enclosing a given bounded subset of $\mathbb{R}^n$.}
The MEB problem is considered significant in various issues of mathematics
(\emph{e.g}.\ geometry and topology, computational and discrete
geometry, non-Euclidean and convex geometry)
with applications including, among others, machine learning, computer graphics and facility location for which
a large amount of articles have been published.
This problem is also referred as
\emph{Euclidean 1-center}, \emph{minimax problem in facility locations},
\emph{Chebyshev radius and Chebyshev center},
\emph{smallest bounding ball}, \emph{minimum bounding sphere},
\emph{minimum spanning ball}, \emph{smallest enclosing sphere}, \emph{etc.}

It is believed that the earliest known statement of the MEB problem was first posed in \emph{circa}\/ 300 B.C.
by Euclid in his work \emph{``Elements''} and specifically in Book IV, Proposition 5 which refers to circumscribe a circle
about a given triangle~(\emph{cf.}\/ \cite{Euclid1482}).
The general case of finding the smallest circle enclosing a given finite set of points in $\mathbb{R}^2$
was first appeared in 1857 by
Sylvester\footnote{\,James Joseph Sylvester (1814 -- 1897), English mathematician.} \cite{Sylvester1857}.
In addition, the earliest known procedure for tackling the above case was also proposed in 1860 by Sylvester~\cite{Sylvester1860}.
The same procedure was independently proposed in 1885 by
Chrystal\footnote{\,George Chrystal (1851 -- 1911), Scottish mathematician.} \cite{Chrystal1885}.
Thereafter, various algorithms for determining the MEB of a set of points in the plane
have been proposed.

It is worth mentioning that, the first optimal linear-time method for fixed dimension
has been presented in 1982 by Megiddo~\cite{Megiddo1982}, (\emph{cf}.\ \cite{Megiddo1983}).
In addition, a simple and very fast in practice randomized algorithm
to solve the problem, also in $\mathbb{R}^2$ and $\mathbb{R}^3$, in expected linear time
has been proposed in 1991 by Welzl~\cite{Welzl1991}.

A relatively recent generalization of Welzl's algorithm for point sets lying in \emph{information-geometric spaces}
as well as an interesting application for solving a \emph{statistical model estimation problem}\/ by computing the
center of a finite set of univariate normal distributions,
have been given in 2008 by Nielsen and Nock~\cite{NielsenN2008}.

As we have already mentioned, Jung in 1901
was first answered to the question of best possible estimate of the radius of a sphere
in $\mathbb{R}^n$ enclosing a bounded subset of $\mathbb{R}^n$
of a given diameter (\emph{cf}.\ Theorem~\ref{JungTh}).
Thereafter, the generalization of MEB in $\mathbb{R}^n$
has been in the spotlight for many years and it remains an important issue of study and further research in the field.
It should be noted that,
the solution of this problem has been proved that is unique in the case of
the \emph{Euclidean geometry} (see, \emph{e.g.}\ Welzl 1991 \cite{Welzl1991}),
the \emph{hyperbolic geometry} (see, \emph{e.g.}\ Nielsen and Hadjeres 2015 \cite{NielsenH2015}),
the \emph{Riemannian positive-definite matrix manifold} (see, \emph{e.g.}\ Lang 1999 \cite{Lang1999}, Nielsen and Bhatia 2013 \cite{NielsenB2013}),
and
in any \emph{Cartan-Hadamard manifold} (see, \emph{e.g.}\  Arnaudon and Nielsen 2013 \cite{ArnaudonN2013}),
\emph{i.e.}\ Riemannian manifold that is complete and simply connected with non-positive sectional curvatures (see, \emph{e.g.}\  Nielsen 2020 \cite{Nielsen2020}).
Also, the above solution is unique in any \emph{Bruhat–Tits space} (see, \emph{e.g.}\ Lang 1999 \cite{Lang1999}),
\emph{i.e.} complete metric space with a semi-parallelogram property,
which includes the Riemannian manifold of symmetric positive definite matrices (see, \emph{e.g.}\  Nielsen 2020 \cite{Nielsen2020}).
On the other hand, the solution of the MEB problem may not be unique in a metric space,
as for instance, in the case of \emph{discrete Hamming metric space} (see, \emph{e.g.}\ Mazumdar \emph{et al.} 2013 \cite{MazumdarPS2013}) which,
in this case, results to NP-hard computation
(\emph{i.e}.\ the complexity class of decision problems that are intrinsically
harder than those that can be solved by a nondeterministic Turing machine in polynomial time)
(see, \emph{e.g.}\ Nielsen 2020 \cite{Nielsen2020}).
For various mathematical approaches to the MEB and related to it problems,
we refer the interest reader to~\cite{Vrahatis2024}.

\subsubsection{Generalizations of Jung's theorem to other spaces and geometries}
In general, Jung's enclosing theorem is considered as a cornerstone in the field that
strongly influencing later developments and generalizations to other spaces and non-Euclidean geometries.
A few examples are as follows.

A generalization of Jung's theorem to \emph{Minkowski spaces}\/
has been proved in 1938 by Bohnenblust \cite{Bohnenblust1938}.
A simpler proof of Bohnenblust's result has been given in 1955 by Leichtweiss \cite{Leichtweiss1955}.
The values of the \emph{expansion constant}\/ and \emph{Jung's constant}\/ determined by the geometric properties
of a Minkowski space have been studied in 1959 by Gr\"{u}nbaum~\cite{Gruenbaum1959}.
Generalizations of Jung's theorem for a \emph{pair of Minkowski spaces}\/
have been given in 2006 by Boltyanski and Martini~\cite{BoltyanskiM2006}.
Particularly, they gave generalizations of Jung’s theorem for the case where there are two Minkowski
metrics in $\mathbb{R}^n$, one for the diameter of a set $P\in \mathbb{R}^n$, and the other for
the radius of the Minkowski ball containing $P$.
An one-to-one connection between the \emph{Jung constant}\/ determined by Jung's theorem, \emph{i.e.}
the maximal ratio of the circumradius and the diameter of a body, in an
arbitrary Minkowski space and the maximal \emph{Minkowski asymmetry}\/ of the
complete bodies within that space has been stated
in 2017 by Brandenberg and Gonz\'{a}lez Merino~\cite{BrandenbergGM2017}.

An \emph{infinity dimensional Hilbert space}\/ extension of Jung's theorem and its application to \emph{measures of noncompactness}\/
have been proposed in 1984 by Dane\u{s}~\cite{Danes1984}.
A complete characterization of the \emph{extremal subsets of Hilbert spaces}, which constitutes
an infinite-dimensional generalization of the Jung theorem, has been proposed in 2006 by
Nguen-Khac and Nguen-Van~\cite{Nguen-KhacN-V2006}.
A generalization of Jung's theorem to convex regions on the \emph{$n$-dimensional spherical surface}\/
has been proposed in 1946 by Santal\'{o}~\cite{Santalo1946}.
The determination of the \emph{Jung diameter}\/ of the \emph{complex projective plane}\/ has be given and
lower bounds for all \emph{complex projective spaces}\/ have been obtained in 1985 by Katz~\cite{Katz1985}.
A generalization of Jung's theorem has been proved in 1985 by Dekster \cite{Dekster1985}
for which estimates are given for the side-lengths of certain inscribed simplices.
The obtained result has been extended to the \emph{spherical and hyperbolic spaces}\/
by Dekster in 1995 \cite{Dekster1995}
and to a class of \emph{metric spaces of curvature bounded above}\/
that includes \emph{Riemannian spaces} also by Dekster in 1997~\cite{Dekster1997}.

An analogue of Jung's theorem for \emph{Alexandrov spaces of curvature bounded above}\/
has been presented in 1997 by Lang and Schroeder~\cite{LangS1997}.
Furthermore, \emph{combinatorial generalizations}\/ of Jung's theorem
have been presented and
the \emph{``fractional''} and \emph{``colorful''} versions of this theorem
have been proved in 2013 by Akopyan~\cite{Akopyan2013}.

\subsection{Simplex thickness}\label{subsec:SimThick}

\begin{remark}
The thickness of a simplex is the dimensionless quantity that it is determined by
the ratio of the minimal distance from its barycenter to its boundary, to that of the diameter of the simplex.
In general, it provides a measure for the quality or how well shaped a simplex is.
\end{remark}

\begin{lemma}\label{LemThic}
Let $\sigma^m$
be an $m$-simplex in\/ ${\mathbb{R}}^n$ $(n \geqslant m)$ with diameter ${\rm diam}(\sigma^m)$
and boundary ${\vartheta}{\sigma}^m$. Assume that
$\kappa^m$ is its barycenter and let
$\sigma^{m}_{\neg i}$ be the $i$-th $(m-1)$-face of $\sigma^m$.
Then, the barycentric inradius of $\sigma^m$ is given by:
\begin{equation}\label{eq:LemInr}
\beta_{\rm inr}^m = \min_{0 \leqslant i \leqslant m}\Bigl\{ \min_{x \in {\sigma_{\neg i}^{m}}} \|\kappa^m-x\| \Bigr\},
\end{equation}
and the thickness of $\sigma^m$ is as folows:
\begin{equation}\label{eq:LemThic}
{\theta}({\sigma}^m) = \min_{0 \leqslant i \leqslant m}\Bigl\{ \min_{x \in {\sigma_{\neg i}^{m}}} \|\kappa^m-x\| \Bigr\} / {\rm diam}(\sigma^m).
\end{equation}
\end{lemma}
\begin{proof}
Since it holds that ${\vartheta}\kern0.02cm$${\sigma}^m = \cup _{i=0}^{m} \sigma^{m}_{\neg i}$,
the proof can be easily obtained by companying Eq.~(\ref{eq:BarInr}) and Eq.~(\ref{eq:thic}).
\end{proof}

\begin{lemma}\label{lem:3.2}
Assume that $\sigma^m = [{\upsilon}^0, {\upsilon}^1, \ldots, {\upsilon}^m]$
is an $m$-simplex in\/ ${\mathbb{R}}^n$ $(n \geqslant m)$. Let
$\kappa^m$ be the barycenter of $\sigma^m$ and let
$\kappa_i^m$ be the barycenter of the $i$-th $(m-1)$-face $\sigma^{m}_{\neg i}$ of $\sigma^m$. Then for each $i=0, 1, \ldots, m$ it holds that:
\begin{equation}\label{eq:3.2}
\|\kappa^{m} - \kappa^{m}_i\| = \frac{1}{m(m + 1)} \Biggl(m\v j0im  \nr {\upsilon}i{\upsilon}j^2-
    \sum_{\substack{p=0\\p\ne i}}^{m-1}
      \v q{p+1}im  \nr {\upsilon}p{\upsilon}q^2\Biggr)^{1/2}.
\end{equation}
\end{lemma}
\begin{proof}
The proof can be easily obtained using Eq.~(\ref{eq:Km4}) and Eq.~(\ref{eq:3.1}).
\end{proof}

\begin{definition}
The quantity
\begin{equation}\label{eq:EstInr}
\tilde{\beta}_{\rm inr}^m = \min_{0 \leqslant i \leqslant m}\|\kappa^{m} - \kappa^{m}_i\|,
\end{equation}
is called \emph{estimation of the barycentric inradius}\/ $\beta_{\rm inr}^m$ of $\sigma^m$, while the quantity
\begin{equation}\label{eq:EstThic}
\tilde{\theta}({\sigma}^m) = \min_{0 \leqslant i \leqslant m}\|\kappa^{m} - \kappa^{m}_i\|/ {\rm diam}(\sigma^m),
\end{equation}
is called \emph{estimation of the thickness}\/ ${\theta}({\sigma}^m)$ of $\sigma^m$.
\end{definition}

\begin{remark}
Obviously, the quantities $\tilde{\beta}_{\rm inr}^m$ and $\tilde{\theta}({\sigma}^m)$ constitute upper bounds for
the quantities $\beta_{\rm inr}^m$ and ${\theta}({\sigma}^m)$, respectively.
\end{remark}

\begin{lemma}\label{lem:inr-reg}
Assume that $\sigma ^m$ is a regular $m$-simplex in ${\mathbb{R}}^n$ with diameter\/ ${\rm diam}(\sigma^m)$,
then the barycentric inradius $\beta_{\rm inr}^m$ is given by:
\begin{equation}\label{eq:equi-inr}
\beta_{\rm inr}^m = \bigl[{2 m (m+1)}\bigr]^{-1/2} {\rm diam}(\sigma^m).
\end{equation}
\end{lemma}
\begin{proof}
Since the lengths of the $m$ edges that have the common vertex ${\upsilon}^i$
as well as the lengths of the $\binom{m}{2}$ edges of the $i$-th $(m-1)$-face $\sigma^{m}_{\neg i}$ are equal to ${\rm diam}(\sigma^m)$,
by applying Lemma~\ref{lem:3.2} we obtain
\begin{align*}
\beta_{\rm inr}^m = &\frac{1}{m (m+1)}\left[m^2 {\rm diam}(\sigma^m)^2 -\binom{m}{2} {\rm diam}(\sigma^m)^2 \right]^{1/2}\\
                  = &\frac{1}{m (m+1)}\left[ \frac{m (m+1)}{2}\right]^{1/2}{\rm diam}(\sigma^m)
                  = \left[\frac{1}{2 m (m+1)}\right]^{1/2} {\rm diam}(\sigma^m).
\end{align*}
Thus the lemma is proved.
\end{proof}

\begin{remark}
The simplex thickness gives a \emph{quality measure}\/ of a simplex.
This is important, among others, to piecewise linear approximations of smooth
mappings in fixed point theory and in complementarity theory.
Also, it is significant to simplicial and continuation methods
for approximating fixed points as well as in locating and computing periodic orbits of nonlinear mappings and
solutions of systems of nonlinear algebraic and/or transcendental
equations. In addition, in general, analysis of simplices characteristics in $\mathbb{R}^n$
is useful in the implementation of high dimensional elements of the well known and widely used finite element method
that gives a formalism for deriving discrete methods for approximating solutions of differential equations
(see, \emph{e.g.}~\cite{BanhelyiCH2020,BoissonnatDGLW2021,BoissonnatW2022,BrennerS2008,Kojima1978,%
Saigal1979,Vrahatis1995,Whitehead1940}).
\end{remark}

\begin{definition}[$\!\!$\cite{Gale1953}]
A set $P$\/ is said to be \emph{inscribed}\/ in an $n$-simplex $\sigma^n$ in $\mathbb{R}^n$ if $P\subset \sigma^n$ and $P$ intersects every $(n-1)$-face of $\sigma^n$.
\end{definition}

\begin{remark}
The following important theorem states that:
``\emph{The regular $n$-simplex whose inscribed sphere has diameter 1
will cover any $n$-dimensional set of diameter~1}''~\cite[p.222]{Gale1953}.
\end{remark}

\begin{theorem}[Gale's theorem (1953)~\cite{Gale1953}]\label{GaleTh}
Assume that $P$ is a closed subset of $\mathbb{R}^n$ of diameter ${\rm diam}(P)=1$. Then $P$\/ can be inscribed in a regular $n$-simplex $\sigma^n$ of diameter
${\rm diam}(\sigma^n) \leqslant\bigl[n(n+1)/2\bigr]^{1/2}$.
\end{theorem}

\begin{lemma}\label{lem:equi-diam}
The diameter of the regular $n$-simplex $\sigma^n$ in\/ $\mathbb{R}^n$ whose the inscribed sphere has diameter 1, is given by
${\rm diam}(\sigma^n) = \bigl[{n(n+1)/2}\bigr]^{1/2}$.
\end{lemma}
\begin{proof}
Since the diameter of the inscribed sphere is $2 \beta_{\rm inr}^n =1$, the proof follows from Eq.~(\ref{eq:equi-inr}).
\end{proof}

\begin{remark}
Gale in~\cite[p.224]{Gale1953} has pointed out that, since $\sqrt{n(n+1)/2}$ is the diameter
of the regular $n$-simplex whose inscribed sphere has diameter~1,
the inequality in Theorem~\ref{GaleTh} is the best possible.
Also, he applied Theorem~\ref{GaleTh} to the case $n = 2$ and he showed that every set in $\mathbb{R}^2$
of diameter~1 is the union of three sets each of diameter
less than $\sqrt{3}/2 <1$. This result is related to the conjecture of
Borsuk\footnote{\,Karol Borsuk (1905 -- 1982), Polish mathematician.}
that states:
``\emph{Every set in $\mathbb{R}^n$ of diameter 1 is the union of $n+1$ sets each of diameter less than~1}''~\cite{Borsuk1933}.
\end{remark}

\subsubsection{Relations between the circumradius, inradius, diameter and width}
In 1958 Eggleston~\cite{Eggleston1958} pointed out that
there are twelve possible inequalities between any two of the four following characteristics of a convex set $P$ in $\mathbb{R}^n$
(containing more than a single point).

Namely, the characteristics are:
(a)~\emph{circumradius} $\rho_{\rm cir}^n$ (\emph{i.e}.\ the radius of the smallest spherical surface enclosing the set),
(b)~\emph{inradius} $\rho_{\rm inr}^n$ (\emph{i.e}.\ the radius of the greatest spherical surface which is contained in the set),
(c)~\emph{diameter} ${\rm diam}(P)$ (\emph{i.e}.\ the maximal distance of any two points of the set)  and
(d)~\emph{width} ${\rm wid}(P)$ (\emph{i.e}.\ the minimum distance of two parallel supporting hyperplanes of the set).

The twelve possible inequalities are categorized as follows. Specifically, the first four trivial inequalities are as follows:
\begin{subequations}\label{eggl-ineq:1}
\begin{align}
\rho_{\rm inr}^n &\leqslant \rho_{\rm cir}^n , & \rho_{\rm inr}^n  & \leqslant \frac{1}{2} {\rm wid}(P), \label{eggl-ineq:c1a}\\[0.3cm]
{\rm diam}(P) &\leqslant 2 \rho_{\rm cir}^n , & {\rm wid}(P)  & \leqslant {\rm diam}(P). \label{eggl-ineq:c1b}
\end{align}
\end{subequations}
The above inequalities imply the following two ones:
\begin{equation}\label{eggl-ineq:2}
\kern0.4cm\rho_{\rm inr}^n \leqslant \frac{1}{2} {\rm diam}(P), \kern1.6cm
 {\rm wid}(P) \leqslant 2 \rho_{\rm cir}^n.
\end{equation}
Furthermore, Eggleston pointed out that since in the case where $\lambda$ is not infinite
there are no inequalities of the following four expressions:
\begin{subequations}\label{eggl-ineq:3}
\begin{align}
\rho_{\rm cir}^n &\leqslant \lambda\,\rho_{\rm inr}^n, & \rho_{\rm cir}^n  & \leqslant \lambda\,{\rm wid}(P), \label{eggl-ineq:c3a}\\[0.3cm]
{\rm diam}(P) &\leqslant \lambda\,\rho_{\rm inr}^n, & {\rm diam}(P)  & \leqslant \lambda\,{\rm wid}(P), \label{eggl-ineq:c3b}
\end{align}
\end{subequations}
and the remaining two inequalities are quite interesting and have the forms:
\begin{equation}\label{eggl-ineq:4}
\kern0.4cm\rho_{\rm cir}^n \leqslant \lambda\,{\rm diam}(P), \kern1.6cm
{\rm wid}(P) \leqslant \lambda\,\rho_{\rm inr}^n.
\end{equation}
Particularly,
the proof of the first of the above inequalities~(\ref{eggl-ineq:4}) that has been given by Eggleston (\emph{cf}.\ \cite[Th.~49]{Eggleston1958})
is related to \emph{Jung's theorem} (\emph{cf}.\ Theorem \ref{JungTh}), while his proof of the second inequality (\emph{cf}.\ \cite[Th.~50]{Eggleston1958})
is related to \emph{Steinhagen's theorem}\/ that it has been given in 1922 by Steinhagen~\cite{Steinhagen1922}
(\emph{cf}.\ Theorem~\ref{SteinhagenTh}).

In addition, it is worth mentioning that, in 1987
Perel${}^\prime\!$man\footnote{\,Grigorii Yakovlevich Perel${}^\prime\!$man (b.\ 1966), Russian mathematician.}
in his publication \emph{``$k$~radii of a compact convex body''}~\cite{Perelman1987}
considered the \emph{internal}\/ and \emph{external}\/ $k$-radii
of a compact  (\emph{i.e}.\ closed and bounded) convex body $P$ in $\mathbb{R}^n$ that are particular cases of the
\emph{diameters}\/ used in \emph{{B}ern\u{s}te\u{\i}n and Kolmogorov approximation theory}, respectively.
Specifically, the \emph{internal $k$-radius}\/ $r_k(P)$ for $k = 1, 2, \ldots, n$ of $P$
is defined as the radius of the greatest $k$-dimensional sphere contained in $P$, while
the \emph{external $k$-radius}\/ $R_k(P)$ for $k = 1, 2, \ldots, n$ of $P$
is defined as the smallest radius of the solid cylinder containing $P$
with a $(n + 1 - k)$-dimensional spherical cross section and
a $(k - 1)$-dimensional generator.

Perel${}^\prime\!$man in~\cite{Perelman1987} and independently Pukhov in \cite{Pukhov1979} (\emph{cf.}~\cite{Gonzalez-Merino2017})
studied the
relation between these internal and external $k$-radii measures, and proved the following theorem:
\begin{theorem}[Perel${}^\prime\!$man-Pukhov theorem~\cite{Perelman1987}]\label{PerelPukhTh}
Let $P$ be a compact convex body in $\mathbb{R}^n$, then it holds that:
\begin{equation}\label{PerelPukhEq}
\frac{R_k(P)}{r_k(P)} \leqslant k + 1, \quad 1 \leqslant k \leqslant n ,
\end{equation}
where $R_k(P)$ and  $r_k(P)$ are the external and internal $k$-radii of $P$, respectively.
\end{theorem}

\begin{remark}
Perel${}^\prime\!$man in~\cite{Perelman1987} pointed out, among others, that:
(a)
The equality $r_k(P) = R_k(P)$ holds for certain bodies, \emph{e.g}.\ spheres,
(b)
Inequality~(\ref{PerelPukhEq}) is also valid for finite $r_k(P)$ without the assumption that the convex body $P$ is bounded,
(c)
The upper bound of Theorem~\ref{PerelPukhTh} is not a least upper bound
for any $1 \leqslant k \leqslant n$
and the estimate given in Inequality~(\ref{PerelPukhEq}) is extremely ``rough'' (non-sharp),
(d)
Precise estimates of $R_k(P)/r_k(P)$ for the cases $k = 2, 3, \ldots, n - 1$
are, evidently, unknown,
(e) For the cases $k = 1$ and $k = n$,
the maximum of $R_k(P)/r_k(P)$ can be reached on a regular $n$-simplex,
(f)
Bounds which are, in fact, least upper bounds for $R_k(P)/r_k(P)$ is the most interesting, since the bound for $R_{\ell}(P)/r_k(P)$
for $k \leqslant \ell$  follows from it,
(g)
Least upper bounds for the cases $R_1(P)/r_1(P)$ and $R_d(P)/r_d(P)$ can be obtained
by Jung's theorem~\cite{Jung1901} and the generalization of the
Blaschke\footnote{\,Wilhelm Johann Eugen Blaschke (1885 -- 1962), Austrian mathematician.}
theorem~\cite{Blaschke1914},
due to
Steinhagen\footnote{\,Paul Steinhagen, German mathematician, Dr. phil. Universit\"{a}t Hamburg 1920.}
\cite{Steinhagen1922}, respectively.
\end{remark}

\begin{theorem}[Steinhagen's theorem (1922)~\cite{Steinhagen1922}]\label{SteinhagenTh}
Let $P$ be a bounded convex subset of\, $\mathbb{R}^n$ of inradius $\rho_{\rm inr}^n$.
Then, the width ${\rm wid}(P)$ satisfies:
\begin{equation}\label{eggl-ineq:Jungste}
{{\rm wid}(P)} \leqslant
\begin{cases}
{2 \sqrt{n}\, \rho_{\rm inr}^n,} &{\text{if}}\ {n}\,\ {\text{is odd}},\\[0.3cm]
\bigl[{2(n+1)/\sqrt{n+2}\,\bigr]\, \rho_{\rm inr}^n,} &{\text{if}}\ {n}\,\ {\text{is even}}.
\end{cases}
\end{equation}
\end{theorem}

\begin{remark}
The above statement of the theorem
gives an upper bound for ${\rm wid}(P)$ in terms of $\rho_{\rm inr}^n$
and we use this statement in order its corresponding  inequality to be compatible with the form of the above
referred secoond inequality in (\ref{eggl-ineq:4}).
The original statement of the theorem refers to the lower bound for $\rho_{\rm inr}^n$ in terms
of ${\rm wid}(P)$ and has the following equivalent form (\emph{cf}.\ \cite[Th.~50]{Eggleston1958}).
\begin{theorem}[Steinhagen's theorem (1922)~\cite{Steinhagen1922}]\label{SteinhagenTh2}
Let $P$ be a bounded convex subset of\, $\mathbb{R}^n$ of width ${\rm wid}(P)$.
Then, the inradius $\rho_{\rm inr}^n$ satisfies:
\begin{equation*}\label{eq:steinh}
{\rho_{\rm inr}^n} \geqslant
\begin{cases}
\bigl(2 \sqrt{n}\bigr)^{-1}\, {{\rm wid}(P)}, & {\text{if}}\ {n}\,\ {\text{is odd}},\\[0.3cm]
\bigl[\sqrt{n + 2}/(2n + 2)\bigr]\, {{\rm wid}(P)}, & {\text{if}}\ {n}\,\ {\text{is even}}.
\end{cases}
\end{equation*}
\end{theorem}
\end{remark}

\begin{remark}
The Steinhagen theorem has been proved  in 1914 for $n = 2$ by Blaschke~\cite{Blaschke1914}.
An additional proof of Steinhagen's theorem has been given in 1936 by Gericke~\cite{Gericke1936}.
Furthermore, in 1946 Santal\'{o}~\cite{Santalo1946}
gave a \emph{generalization of Jung's and Steinhagen's theorems} to convex regions on the \emph{$n$-dimensional spherical surface}.
A geometrical proof of the above results for $n = 2$  has been given in 1944, also by Santal\'{o}~\cite{Santalo1944}.
It is worth noting that, Santal\'{o}'s result about the generalization of the Steinhagen's theorem, also generalizes
the well-known \emph{Robinson's theorem} that it has been proved by Robinson in 1938~\cite{Robinson1938}.
In 1992 Henk~\cite{Henk1992} by studying the relation between circumradius and diameter of a convex
body which has been proposed by Jung's theorem, gave
a natural \emph{generalization of Jung's theorem}.
Also, in 1993 Betke and Henk~\cite{BetkeH1993} by analyzing the relation between inradius and
width of a convex set that has been established by Steinhagen's theorem proposed respectively
a natural \emph{generalization of Steinhagen's theorem}.
\end{remark}

\begin{remark}
The circumradius, inradius and width of the regular $n$-simplex $\sigma^n$ in ${\mathbb{R}}^n$
with diameter 1 are given as follows.
The circumradius $\rho_{\rm cir}^n $ of $\sigma^n$, which in the case of regular simplices is
the same with the barycentric circumradius $\beta_{\rm cir}^n$
is $\rho_{\rm cir}^n \equiv  \beta_{\rm cir}^n  = \sqrt{n/(2n+2)}$
(\emph{cf}.\ Lemma~\ref{lem:circ-reg} and  Jung's Theorem \ref{JungTh}).
Similarly,
the inradius $\rho_{\rm inr}^n$ of $\sigma^n$ which is the same with
the barycentric inradius $\beta_{\rm inr}^n$ is given by
$\rho_{\rm inr}^n  \equiv \beta_{\rm inr}^n  = 1 / \sqrt{2 n (n + 1)}$
(\emph{cf}.\ Lemma \ref{lem:inr-reg}).
Also,
the width of $\sigma^n$ is
${\rm wid}(\sigma^n) \approx \sqrt{2 / n}$
and specifically ${\rm wid}(\sigma^n) = \sqrt{2 / (n +1)}$ if $n$ is odd and
${\rm wid}(\sigma^n) = \sqrt{2 (n + 1)}/\sqrt{n (n + 2)}$ if~$n$ is even
(\emph{cf}.\ \cite{Alexander1977,GritzmannK1992,Har-PeledR2023}).
\end{remark}

Based on Steinhagen's theorem we give a proof
for the above referred exact bound on the width
of the regular simplex in ${\mathbb{R}}^n$ with diameter 1.

\begin{theorem}[Width of the regular simplex]\label{thm:reg-width}
Let $\sigma^n$ be a regular $n$-simplex in ${\mathbb{R}}^n$ with diameter\/ ${\rm diam}(\sigma^n)=1$,
then the width ${\rm wid}(\sigma^n)$ of $\sigma^n$ is given by:
\begin{equation}\label{eggl-ineq:Jungste-reg}
{{\rm wid}(\sigma^n)} =
\begin{cases}
{\sqrt{2 / (n +1)},} &{\text{if}}\ {n}\,\ {\text{is odd}},\\[0.3cm]
{\sqrt{2 (n + 1)}/\sqrt{n (n + 2)},} &{\text{if}}\ {n}\,\ {\text{is even}}.
\end{cases}
\end{equation}
\end{theorem}
\begin{proof}
The proof follows from Lemma~\ref{lem:inr-reg} and Theorem~\ref{SteinhagenTh}.
The equality holds since $\sigma^n$ is regular (\emph{cf}.\ \cite{Alexander1977} and Remark~\ref{rem:reg-prop}).
\end{proof}

\subsection{Simplex bisection}\label{subsec:SimBis}

\begin{remark}
An $m$-simplex of $\mathbb{R}^n$ is bisected by choosing its longest edge, dividing it in two equal parts and defining two new $m$-simplices.
Each of them will, in turn, be bisected and this process continues for a finite number of iterations.
\end{remark}

\begin{definition}[Simplex bisection~\cite{Kearfott1978,RosenbergS1975,Stenger1975}]
Let
${\sigma}^m_{\langle{0}\rangle} =
[{\upsilon }^0,{\upsilon }^1,\ldots,
{\upsilon }^m ]$
be an $m$-simplex in $\mathbb{R}^n$ $(n \geqslant m)$.
Assume that $[ {\upsilon }^i,{\upsilon }^j ]$
is the longest edge of ${\sigma}^m_{\langle{0}\rangle}$ and
let $M = ( {\upsilon }^i+{\upsilon }^j ) / 2 $
be its midpoint.
Then the {\it bisection\/} of ${\sigma}^m_{\langle{0}\rangle}$ is the order
pair of $m$-simplices
$\bigl\{{{\sigma}^m_{\langle{10}\rangle},\, {\sigma}^m_{\langle{11}\rangle} } \bigr\}$
where:
\begin{align*}
&{\sigma}^m_{\langle{10}\rangle} = [ {\upsilon }^0, {\upsilon }^1, \ldots,
  {\upsilon }^{i-1}, M,{\upsilon }^{i+1}, \ldots,
  {\upsilon }^j, \ldots, {\upsilon }^m ],\\[0.05cm]
&{\sigma}^m_{\langle{11}\rangle} = [ {\upsilon }^0, {\upsilon }^1, \ldots,
  {\upsilon }^{i}, \ldots, {\upsilon }^{j-1}, M,{\upsilon }^{j+1},
  \ldots,{\upsilon }^m ].
\end{align*}
The $m$-simplices ${\sigma}^m_{\langle{10}\rangle}$ and ${\sigma}^m_{\langle{11}\rangle}$
are called \emph{lower simplex}\/ and
\emph{upper simplex}\/ respectively corresponding to
${\sigma}^m_{\langle{0}\rangle}$, while both ${\sigma}^m_{\langle{10}\rangle}$ and ${\sigma}^m_{\langle{11}\rangle}$
are called \emph{elements of the bisection}\/ of ${\sigma}^m_{\langle{0}\rangle}$.
Next, each of the elements of the bisection will, in turn, be bisected and the process will be continued to get a sequence of sets of simplices.
In case where one of the elements is selected then this
element is called \emph{selected $m$-simplex produced after one
 bisection of}\/ ${\sigma}^m_{\langle{0}\rangle}$ and it is denoted by
${\sigma }^m_{\langle{1}\rangle}$.
Assume now that the bisection is applied
with ${\sigma}^m_{\langle{1}\rangle}$ replacing ${\sigma}^m_{\langle{0}\rangle}$
giving thus the ${\sigma}^m_{\langle{2}\rangle}$.
Suppose further that this process continues for $p$
iterations. Then the ${\sigma}^m_{\langle{p}\rangle}$ is called
the \emph{selected $m$-simplex produced after $p$
iterations of the bisection of ${\sigma}^m_{\langle{0}\rangle}$}.
\end{definition}

\begin{theorem}[Kearfott's theorem (1978)~\cite{Kearfott1978}]\label{Th-Kearfott1978}
Let
${\sigma}^m_{\langle{0}\rangle} =
[ {\upsilon }^0,{\upsilon }^1,\ldots, {\upsilon }^m ]$
be an $m$-simplex in $\mathbb{R}^n$ $(n \geqslant m)$
and  let ${\sigma}^m_{\langle{p}\rangle}$ be any $m$-simplex produced after $p$
 bisections of  ${\sigma}^m_{\langle{0}\rangle}$. Then it holds that:
\begin{equation}
 {\rm diam}\bigl({\sigma}^m_{\langle{p}\rangle}\bigr) \leqslant
\bigl(\sqrt{3} / 2 \bigr)^{\lfloor p / m \rfloor } {\rm diam}\bigl({\sigma}^m_{\langle{0}\rangle}\bigr),
\end{equation}
where ${\rm diam}\bigl({\sigma}^m_{\langle{p}\rangle}\bigr)$ and ${\rm diam}\bigl({\sigma}^m_{\langle{0}\rangle}\bigr)$ are the diameters
of ${\sigma}^m_{\langle{p}\rangle}$ and ${\sigma}^m_{\langle{0}\rangle}$  respectively and
$\lfloor p / m \rfloor$ is the largest integer less than or equal to $p/m$.
\end{theorem}

\begin{lemma}\label{Lem-Vrahatis-Bis1986}
Assume that ${\sigma}^m_{\langle{0}\rangle}$, ${\sigma}^m_{\langle{p}\rangle}$,
${\rm diam}\bigl({\sigma}^m_{\langle{0}\rangle}\bigr)$ and\/ ${\rm diam}\bigl({\sigma}^m_{\langle{p}\rangle}\bigr)$ are
as in Theorem~\ref{Th-Kearfott1978}. Let $\kappa^m_{\langle{p}\rangle}$
and $\beta_{\rm cir}^m = \beta_{\rm cir}^m\bigl({\sigma}^m_{\langle{p}\rangle}\bigr)$ be the
barycenter and the barycentric circumradius of~${\sigma}^m_{\langle{p}\rangle}$ respectively. Then for any point
$x$ in ${\sigma}^m_{\langle{p}\rangle}$ it holds that:
\begin{equation}
 \| x - \kappa^m_{\langle{p}\rangle} \| \leqslant \frac{m}{m+1}
\bigl(\sqrt{3} / 2 \bigr)^{\lfloor p / m \rfloor } {\rm diam}\bigl({\sigma}^m_{\langle{0}\rangle}\bigr).
\end{equation}
\end{lemma}
\begin{proof}
By Theorem~\ref{thm:simcov} we obtain that
\[
 \| x - \kappa^m_{\langle{p}\rangle} \| \leqslant \beta_{\rm cir}^m\bigl({\sigma}^m_{\langle{p}\rangle}\bigr) \leqslant
 \frac{1}{m + 1} \Bigl( m^2 \, {\rm diam}\bigl({\sigma}^m_{\langle{p}\rangle}\bigr)^2\Bigr)^{1/2},
\]
and the result follows by Theorem~\ref{Th-Kearfott1978}.
\end{proof}

\begin{lemma}\label{Lem-Vrahatis-Bis21986}
Assume that ${\sigma}^m_{\langle{0}\rangle}$, ${\sigma}^m_{\langle{p}\rangle}$,
 ${\rm diam}\bigl({\sigma}^m_{\langle{0}\rangle}\bigr)$ and\/
 ${\rm diam}\bigl({\sigma}^m_{\langle{p}\rangle}\bigr)$ are
as in Theorem~\ref{Th-Kearfott1978}.
Assume further that ${\rm shor}\bigl({\sigma}^m_{\langle{p}\rangle}\bigr)$ is the length of the shortest edge of ${\sigma}^m_{\langle{p}\rangle}$
and let
$\kappa^m_{\langle{p}\rangle}$
and $\beta_{\rm cir}^m = \beta_{\rm cir}^m\bigl({\sigma}^m_{\langle{p}\rangle}\bigr)$ be the
barycenter and the barycentric circumradius of~${\sigma}^m_{\langle{p}\rangle}$ respectively. Then for any point
$x$ in ${\sigma}^m_{\langle{p}\rangle}$ it holds that:
\begin{equation}
 \| x - \kappa^m_{\langle{p}\rangle} \| \leqslant \frac{m}{m+1}
 \Bigl(
 {\rm diam}\bigl({\sigma}^m_{\langle{p}\rangle}\bigr)^2 -
 \frac{m-1}{2m}\, {\rm shor}\bigl({\sigma}^m_{\langle{p}\rangle}\bigr)^2
 \Bigr)^{1/2}.
\end{equation}
\end{lemma}
\begin{proof}
By Theorem~\ref{thm:simcov} we obtain that
\begin{align*}
\| x - \kappa^m_{\langle{p}\rangle} \|
  & \leqslant \beta_{\rm cir}^m \bigl({\sigma}^m_{\langle{p}\rangle}\bigr)
  \leqslant
 \frac{1}{m + 1} \Bigl( m^2 \, {\rm diam}\bigl({\sigma}^m_{\langle{p}\rangle}\bigr)^2-
\binom{m}{2}\, {\rm shor}\bigl({\sigma}^m_{\langle{p}\rangle}\bigr)^2
 \Bigr)^{1/2}\\[0.2cm]
  & \leqslant
 \frac{1}{m + 1} \Bigl( m^2 \, {\rm diam}\bigl({\sigma}^m_{\langle{p}\rangle}\bigr)^2-
\frac{m(m-1)}{2}\, {\rm shor}\bigl({\sigma}^m_{\langle{p}\rangle}\bigr)^2
 \Bigr)^{1/2}.
\end{align*}
Thus the lemma is proved.
\end{proof}

\begin{remark}
Simplex bisection methods are used for the computation of
solutions of systems of nonlinear algebraic and/or transcendental equations of the form:
\begin{equation} \label{fequ}
F_{n}(x)=\mathit{\Theta^n},\,\, \mbox{ with }\,\,
 F_{n}=(f_1, f_2, \ldots , f_n): {\mathcal D} \subset \mathbb{R}^n \to \mathbb{R}^n,
\end{equation}
where $\mathit{\Theta^n}= ( 0, 0, \ldots , 0 )$ is the origin of $\mathbb{R}^n$ and ${\mathcal D}$ is a bounded domain.
Specifically, an $n$-simplex of $\mathbb{R}^n$ can be bisected by choosing its longest edge,
dividing it in two equal parts and considering two new $n$-simplices.
Each of them or one of them is, in turn, bisected. This procedure continues for a finite number
of iterations for obtaining an approximation $s^*$ of a solution~$s$ of Eq.~(\ref{fequ}).
For the selection of an element of the bisection process in each iteration
a criterion is applied for the existence of a solution of Eq.~(\ref{fequ}) in this element
(see, \emph{e.g.}~\cite{Heindl2016,Mayer2002,MilgromM2018,Stenger1975,%
Vrahatis1986,Vrahatis1988b,Vrahatis1989,Vrahatis1995,Vrahatis2000,Vrahatis2016,Vrahatis2020,Vrahatis2022}).
In general, the bisection methods require only the algebraic signs of the function values
and are of major importance for solving problems with imprecise information.
Notice that the solution of Eq.~(\ref{fequ}) is equivalent to the problem of
computing fixed points of continuous functions in several variables
as well as periodic orbits of nonlinear mappings and similarly, fixed
points of the Poincar\'{e} map on a surface of section (see, \emph{e.g.}~\cite{Vrahatis1995,Vrahatis2022}).
\end{remark}

\begin{definition}
Assume that $\sigma^n $ is an $n$-simplex in $\mathbb{R}^n$
and let ${\rm diam}(\sigma^n)$ and ${\rm shor}(\sigma^n)$ be the diameter and
the length of the shortest edge of $\sigma^n$ respectively.
Assume further that~$s$ is a solution of Eq.~(\ref{fequ}) in $\sigma^n$.
Then the barycenter $\kappa^n$ of $\sigma^n$ is defined
to be an \emph{approximation}\/ of~$s$ and the quantity:
\begin{equation}\label{eq:errorest}
 \varepsilon^n = \frac{n}{n+1} \Bigl(
 {\rm diam}(\sigma^n)^2 -
 \frac{n-1}{2n}\, {\rm shor}(\sigma^n)^2
 \Bigr)^{1/2},
\end{equation}
to be an \emph{error estimate}\/ for $\kappa^n$.
\end{definition}

\begin{lemma}\label{Lem-Vrahatis-Bis31986}
Let ${\sigma}^n_{\langle{p}\rangle}$ be the selected $n$-simplex
produced after $p$ bisections of an $n$-simplex
  ${\sigma}^n_{\langle{0}\rangle}$ in $\mathbb{R}^n$ with diameter
   ${\rm diam}\bigl({\sigma}^m_{\langle{0}\rangle}\bigr)$.
Let ${\rm shor}\bigl({\sigma}^m_{\langle{p}\rangle}\bigr)$ be the length of the shortest edge of ${\sigma}^m_{\langle{p}\rangle}$ and
${\rm diam}\bigl({\sigma}^m_{\langle{p}\rangle}\bigr)$ its diameter.
Then
\[
\varepsilon^n_{\langle{p}\rangle} \leqslant\frac{n}{n+1}
\bigl(\sqrt{3} / 2 \bigr)^{\lfloor p / n \rfloor } {\rm diam}\bigl({\sigma}^n_{\langle{0}\rangle}\bigr).
\]
\end{lemma}
\begin{proof}
Using Eq.~(\ref{eq:errorest}) the proof follows by Lemmata~\ref{Lem-Vrahatis-Bis1986} and \ref{Lem-Vrahatis-Bis21986}.
\end{proof}

\begin{theorem}[Convergence of simplex bisection (1986)~\cite{Vrahatis1986}]
Let ${\sigma}^n_{\langle{p}\rangle}$ be the selected $n$-simplex
produced after $p$ bisections of an $n$-simplex
  ${\sigma}^n_{\langle{0}\rangle}$ in $\mathbb{R}^n$.
Assume that $s$ is a solution of Eq.~(\ref{fequ})
in ${\sigma}^n_{\langle{p}\rangle}$
and let $\kappa^n_{\langle{p}\rangle}$ and $\varepsilon^n_{\langle{p}\rangle}$
be the approximation of~$s$ and the error estimate for $\kappa^n_{\langle{p}\rangle}$
respectively. Then\/
$
\lim_{p\to\infty} \varepsilon^n_{\langle{p}\rangle} = 0$
and\/ $\lim_{p\to\infty} \kappa^n_{\langle{p}\rangle} = s.
$
\end{theorem}
\begin{proof}
Using Eq.~(\ref{eq:errorest}) the proof follows by Lemma~\ref{Lem-Vrahatis-Bis1986}.
\end{proof}

\begin{remark}
A significant related problem (which still remains attractive) refers to the determination
of a positive bound on the ratio of the lengths of the smallest edges of simplices to the lengths of the largest edges,
as the simplex bisection is proceed~\cite{Kearfott1978}.
\end{remark}

\end{document}